\documentclass[a4paper]{gtart}

\usepackage[inner=20mm, outer=20mm, textheight=245mm]{geometry}

\usepackage[toc,page]{appendix}

\usepackage{psfrag, graphicx, subfigure, epsfig,color}

\usepackage{amsmath, amssymb, latexsym, euscript}

\usepackage[colorlinks,citecolor=blue,linkcolor=blue, backref=page]{hyperref}
\usepackage[nameinlink]{cleveref}

\usepackage[margin=1cm]{caption}

\usepackage{mathptmx}
\usepackage{wrapfig}
\usepackage[all,cmtip]{xy}
\usepackage{mathtools}
\usepackage{faktor}
\usepackage{tikz}
\usetikzlibrary{cd}
\usepackage{multicol}
\usepackage{listings}

\usepackage[colorinlistoftodos]{todonotes}
\newcommand*{\defeq}{\mathrel{\vcenter{\baselineskip0.5ex \lineskiplimit0pt
                     \hbox{\scriptsize.}\hbox{\scriptsize.}}}%
                     =}

\DeclareMathOperator{\im}{im}
\DeclareMathOperator{\Aut}{Aut}
\DeclareMathOperator{\SL}{SL}
\DeclareMathOperator{\PSL}{PSL}
\DeclareMathOperator{\tr}{tr}
\DeclareMathOperator{\rank}{rank}
\DeclareMathOperator{\Sym}{Sym}
\DeclareMathOperator{\Hom}{Hom}
\DeclareMathOperator{\coker}{coker}
\DeclareMathOperator{\Diag}{Diag}

\newcommand{\A}{\tt{A}}
\newcommand{\B}{\tt{B}}

\newcommand\prho{\overline{\rho}}

\theoremstyle{plain}
\newtheorem{theorem}{Theorem}
\newtheorem*{theorem*}{Theorem}
\newtheorem{lemma}[theorem]{Lemma}
\newtheorem{proposition}[theorem]{Proposition}
\newtheorem{corollary}[theorem]{Corollary}
\newtheorem{scholion}[theorem]{Scholion}

\theoremstyle{definition}
\newtheorem{definition}[theorem]{Definition}
\newtheorem*{definition*}{Definition}
\newtheorem{example}[theorem]{Example}

\newtheorem{fact}[theorem]{Fact}

\theoremstyle{remark}
\newtheorem{remark}[theorem]{Remark}

\numberwithin{equation}{section}

\usepackage{enumitem}
\newlist{enumRoman}{enumerate}{1}
\setlist[enumRoman]{label=(\Roman*)}

\begin{document}

\title{On the topology of character varieties of once-punctured torus bundles}
\author{Stephan Tillmann and Youheng Yao}

\begin{abstract} 
This paper presents, for the special case of once-punctured torus bundles, a natural method to study the character varieties of hyperbolic 3--manifolds that are bundles over the circle. The main strategy is to restrict characters to the fibre of the bundle, and to analyse the resulting branched covering map. This allows us to extend results of Steven Boyer, Erhard Luft and Xingru Zhang. Both $\SL(2, \mathbb{C})$--character varieties and $\PSL(2, \mathbb{C})$--character varieties are considered. As an explicit application of these methods, we build on work of Baker and Petersen to show that there is an infinite family of hyperbolic once-punctured bundles with canonical curves of $\PSL(2, \mathbb{C})$--characters of unbounded genus.
\end{abstract}

\makeshorttitle


\section{Introduction}
\label{sec:intro}

The first part of this paper extends results of Boyer, Luft and Zhang~\cite{Boyer-algebraic-2002} concerning the $\SL(2, \mathbb{C})$--character variety $X(M_\varphi)$ of $M_\varphi$, a hyperbolic once-punctured torus bundle over $S^1$ with monodromy $\varphi$, and we determine its relationship with the $\PSL(2, \mathbb{C})$--character variety $\overline{X}(M_\varphi).$ Denote a once-punctured torus fibre by $S$ and let $\lambda$ be the associated longitude, i.e. the boundary of a compact core of $S$.
The monodromy induces a polynomial automorphism $\overline{\varphi}\co X(S)\to X(S)$ and we denote its fixed point set by $X_\varphi(S).$ This is the image of the restriction map $r\co X(M_\varphi)\to X(S)$ (\Cref{lem:im(r)}). 
We are interested in the topology of the Zariski components\footnote{Our \emph{varieties} are affine algebraic sets, and to avoid using the word \emph{irreducible} in two different ways, we refer to irreducible components of these sets as Zariski components.} of $X(M_\varphi)$ that contain the characters of irreducible representations. Each such component is one-dimensional (\Cref{pro:dimension_bound}), and we let $X^{\text{irr}}(M_\varphi)$ denote their union. 

Boyer, Luft and Zhang~\cite{Boyer-algebraic-2002} show that the binary dihedral characters in $X(M_\varphi(\lambda))\subset X(M_\varphi)$ are all simple points of $X(M_\varphi)$, and they provide an exact count of them. Namely their number is
 \begin{equation}\label{eq:binary_count}
 \frac{1}{2} |2 + \tr(\varphi_*)| - 2^{{b_1(\varphi)-2}}
 \end{equation}
where ${b_1(\varphi)} = b_1(M_\varphi; \mathbb{Z}_2)$ and $\varphi_*$ represents the automorphism induced by $\varphi$ on first homology of $S.$ We show that these are the only irreducible characters of $M_\varphi$ that restrict to reducible characters of $S$ (\Cref{schol:binary_dihedral}).

Let $\varepsilon\co \pi_1(M_\varphi)\to\mathbb{Z}_2$ be the non-trivial homomorphism that is trivial on $\pi_1(S).$

\begin{theorem}
\label{thm:main}
Suppose $M_\varphi$ is a hyperbolic once-punctured torus bundle. Then each Zariski component of $X^{\text{irr}}(M_\varphi)$ is one-dimensional and $r\co X^{\text{irr}}(M_\varphi)\to X_\varphi(S)$ is a two-fold branched covering map. The ramification points of  $r$ are simple points, namely the binary dihedral characters in $X(M_\varphi(\lambda))\subset X(M_\varphi)$. The fibres of $r$ correspond to the orbits of the action of $\langle \varepsilon\rangle  \le H^1(M_\varphi,\mathbb{Z}_2 ).$
 In particular,
\[
\chi(X^{\text{irr}}(M_\varphi)) = 2\chi(X_\varphi(S)) - \frac{1}{2}|2 + \tr(\varphi_*)| + 2^{{b_1(\varphi)-2}}
\]
where $\chi$ denotes the Euler characteristic in the above equation.
\end{theorem}

Implicit in the above statement is that $r(X(M_\varphi)) = r(X^{\text{irr}}(M_\varphi)).$ This is shown as follows.
Let $C$ be a Zariski component containing only reducible representations. Then $C$  is equivalent to $\mathbb{C}$ and the restriction of $r$ to $C$ is constant.
We apply a criterion due to Heusener and Porti~\cite{Heusener-deformation-2005} to show that $C\cap X^{\text{irr}}(M_\varphi) \neq \emptyset,$ and that each point of intersection is contained on a unique curve in $X^{\text{irr}}(M_\varphi).$ Moreover, the intersection point is a smooth point of both curves and the intersection is transverse
 (\Cref{pro:intersection_red_irrep}). 
 
In our applications, we either use  genus bounds that follow from the above theorem (\Cref{cor:genus_bounds}) or well-known results linking genus to the Newton polygon (see \Cref{sec:Genus and the Newton Polygon}). 
The proof of \Cref{thm:main} is organised as follows. We first show that no ramification point in $X^{\text{irr}}(M_\varphi)$ is the character of a reducible representation (\Cref{subsec:Reducible characters}), and then 
characterise the fibres as orbits of $\varepsilon$ 
and complete the proof with \Cref{cor:irreps in fixed point set}.

The map $X^{\text{irr}}(M_\varphi) \to \overline{X}^{\text{irr}}(M_\varphi)$ can be viewed as the quotient map
\[X^{\text{irr}}(M_\varphi) \to X^{\text{irr}}(M_\varphi)/H^1(M_\varphi,\mathbb{Z}_2 )\subseteq \overline{X}^{\text{irr}}(M_\varphi)\]
There is a natural splitting
$H^1(M_\varphi,\mathbb{Z}_2 ) = \langle \varepsilon\rangle \oplus H$, where $H$ is the subgroup generated by all homomorphisms that are only non-trivial on $\pi_1(S).$ We may view $H \le H^1(S,\mathbb{Z}_2 )\cong \mathbb{Z}_2\oplus \mathbb{Z}_2.$ 
\Cref{thm:main} and the analogous result for projective representation (\Cref{pro:degrees_PSLrestriction}) imply the following:

\begin{corollary}
\label{cor:diagram}
Suppose $M_\varphi$ is a hyperbolic once-punctured torus bundle. 
We have the following commutative diagram of maps whose degrees are indicated in the diagram:
\[
\xymatrix{
X^{\text{irr}}(M_\varphi) \ar[d]^{2^{{b_1(\varphi)}}:1}_{q_1} \ar[r]^{r}_{2:1} 	& X_\varphi(S) \ar[d]^{2^{{b_1(\varphi)}-1}:1}_{q_2} \ar@{}[r]|-*[@]{\subset} 	& X(S) \ar[d]^{4:1}_{q_2} \ar@{}[r]|-*[@]{\cong} & \mathbb{C}^3\\
\overline{X}^{\text{irr}}(M_\varphi) \ar[r]^{\overline{r}}_{} 		& \overline{X}_\varphi(S)  \ar@{}[r]|-*[@]{\subset}	& \overline{X}(S)	
}		
\]
Here, each of the maps $r,$ $q_1,$ $q_2$ is a branched covering map. 
The map $q_1$ is the quotient map associated with the action of $H^1(M_\varphi,\mathbb{Z}_2 ) = \langle \varepsilon\rangle \oplus H,$ and  the map $q_2$ is the quotient map associated with the action of $H^1(S,\mathbb{Z}_2 )$ on $X(S).$ 
The map $r$ is the quotient map associated with the action of $\langle \varepsilon \rangle.$ 
The orbits of the action of $H^1(S,\mathbb{Z}_2)$ on $X_\varphi(S)$ are the same as the orbits of $H$ except when 
$b_1(\varphi)=2$ and one of the coordinate axes (with respect to the natural identification $X(S)= \mathbb{C}^3$) is contained in $X_\varphi(S).$
\end{corollary}

The fibres of the map $\overline{r}$ may have size one, two or four, arising from irreducible representations with non-trivial centraliser. The restriction $\overline{r}\co\overline{X}  \to \overline{X}_\varphi(S)$ to a Zariski component $\overline{X}\subseteq \overline{X}^{\text{irr}}(M_\varphi)$ is of degree one, except for a special situation when $b_1(\varphi)=2$ and the degree is two (see \Cref{pro:degrees_PSLrestriction}).

\begin{corollary}
\label{cor:iso for b1=1}
If $b_1(\varphi)=1,$ then we have a birational isomorphism $\overline{X}^{\text{irr}}(M_\varphi) \cong X_\varphi(S).$
\end{corollary}

\begin{proof}
Since $H_1(M; \mathbb{Z}_2) = \mathbb{Z}_2$, we have $H^2(\pi_1(M); \mathbb{Z}_2)\cong 0$ and hence every representation into $\PSL(2, \mathbb{C})$ lifts to a representation into $\SL(2, \mathbb{C}).$ See \Cref{sec:basic_char} for details. 
Hence we obtain the claimed birational isomorphism from the commutative diagram in \Cref{cor:diagram} since the fibres of each of the surjective maps $q_1$ and $r$ are the orbits of the action of $\langle \varepsilon\rangle.$
\end{proof}

The hypothesis of \Cref{cor:iso for b1=1} is equivalent to $\varphi_*$ mapping to an element of order three under the natural epimorphism $\SL(2,\mathbb{Z}) \to \SL(2,\mathbb{Z}_2)\cong\Sym(3)$ (see \Cref{sec:prelim_bundles}).

\Cref{sec: char of once-punctured torus bundles} also contains a number of interesting facts about $X(M_\varphi)$ and $X_\varphi(S)$ that follow from our set-up. For instance, $M_\varphi$ always admits an irreducible representation whose restriction to the fibre has image the quaternionic group. If $H_1(X,\mathbb{Z}_2)\cong \mathbb{Z}_2^3$, then this character always lies on one (or three) special line(s) in $X_\varphi(S)$. 
\Cref{sec:Projective characters} provides analogous results for the  $\PSL(2, \mathbb{C})$--character variety $\overline{X}(M_\varphi)$ and the associated fixed point set $\overline{X}_\varphi(S)\subset \overline{X}(S).$ We use a dimension bound due to Thurston, for which we provide a proof (\Cref{pro:dimPSL_Thurston}) and which is not implied by the dimension bounds in the setting of more general Lie groups due to Falbel and Guilloux~\cite{Falbel-dimension-2017} and Porti~\cite{Porti-dimension-2020}.

In \Cref{sec:examples}, we apply our techniques to infinite families of once-punctured torus bundles and a number of small examples. We follow the standard convention letting $\A, \B \in \SL(2, \mathbb{Z})$ correspond to right-handed Dehn twists (see~\Cref{sec:prelim_bundles}).

\textbf{First family.} Baker and Petersen~\cite{baker-character-2013} determined the genera of the character varieties of the family $M_n$ of once-punctured torus bundles with monodromy $\varphi_n = \A\B^{n+2}$. We note that $b_1(\varphi_n) = 1$ if $n$ is odd and $b_1(\varphi_n) = 2$ if $n$ is even. In the case where $n\ge 3$ is odd, we show that $X_{\varphi_n}(S)$ is a curve of genus zero, and hence \Cref{cor:iso for b1=1} implies that the genus of $\overline{X}^{\text{irr}}(M_{n})$ is zero. This agrees with the computation in~\cite{baker-character-2013}. See \Cref{sec:first_family}.

\textbf{Second family.} Let $N_n$ be the once-punctured torus bundle with monodromy $\psi_n = \A\B^{n+2}A$. We note that $b_1(\psi_n) = 2$ if $n$ is odd and $b_1(\psi_n) = 3$ if $n$ is even. 
In the case where $n=2k+1\ge 1$ is odd, we show that $X_{\psi_n}(S)$ is a curve of genus $k$ and the genus of $\overline{X}^{\text{irr}}(N_{n})$ is zero. See \Cref{sec:second_family}.

\textbf{Third family.} Let $L_n$ be the once-punctured torus bundle with monodromy $\omega_n = \A^2\B^{n+2}\A$. We note that $b_1(\omega_n) = 1$ if $n$ is odd and $b_1(\omega_n) = 2$ if $n$ is even. 
In the case where $n=2k+1\ge 3$ is odd, we show that $X_{\omega_n}(S)\cong \overline{X}^{\text{irr}}(L_{n})$ has $\frac{n+3}{2}$ components, one of which is a curve of genus $k$ and the others have genus $0.$ The curves of genus $k$ are all canonical. This answers a question posed by Alan Reid.
See \Cref{sec:third_family}.

\textbf{Small examples.} We summarise computational results for monodromies that are short words in $\A$ and $\B$, or in $\A$ and $\B^{-1}$, in \Cref{sec:computations}. These were executed with \texttt{Singular}~\cite{Singular}\rm.

We remark that in all examples produced in this paper, each fixed-point set $X_{\varphi}(S)$ has the property that there is at most one Zariski component of positive genus and each Zariski component is birationally equivalent to an elliptic or a hyperelliptic curve.


{\bf Acknowledgements.}
Research of the first author is partially supported by an Australian Research Council Future Fellowship (project number FT170100316). Research of the second author was supported by an Australian Mathematical Sciences Institute (AMSI) Vacation Research Scholarship during the initial stages of this project.
The second author thanks Grace Garden for helpful conversations during this time.


\section{Preliminaries}
\label{sec:Preliminaries}

This section collects some basic facts about character varieties of 3--manifolds in \Cref{sec:basic_char} and about hyperbolic once-punctured torus bundles in \Cref{sec:prelim_bundles}.
We also recall facts about mutation in \Cref{sec:mutation} and the character variety of the free group of rank two in \Cref{sec:charF2}.
None of the material in this section is new, and it mainly serves to set up notation. In \Cref{app:thurston} we also include a statement and proof of a dimension bound based on \cite[Theorem 5.6]{Thurston-notes} in Thurston's notes. For varieties of representations into $\SL(2,\mathbb{C})$ this was given by Culler and Shalen~\cite[Proposition 3.2.1]{Culler-Shalen-varieties-1983}, but we could not find the statement of \Cref{pro:dimPSL_Thurston} for the case of representations into $\PSL(2,\mathbb{C})$ in the literature.


\subsection{Basic facts about character varieties}
\label{sec:basic_char}

We start by summarising the discussion found in \cite[pp.512-513]{MS-degenerations-III-1988}. Let $M$ be a compact, connected and orientable 3--manifold. Let $R(M)$ be the variety of representations of $\pi_1(M)$ into $\SL(2, \mathbb{C})$, and $X(M)$ be the associated character variety. Similarly, let $\overline{R}(M)$ and $\overline{X}(M)$ be the respective varieties of representations into $\PSL(2, \mathbb{C})$ and their characters.

The quotient map $q\co \SL(2, \mathbb{C}) \to \PSL(2, \mathbb{C})$ induces a regular map 
$q_\star\co R(M) \to \overline{R}(M),$ and its image is a union of topological components of $\overline{R}(M).$ Indeed, $R(M) \to q_\star(R(M))$ is a regular covering map with group $H^1(\pi_1(M); \mathbb{Z}_2).$ Following \cite{MS-degenerations-III-1988}, we let $Q(M) = R(M)/H^1(\pi_1(M); \mathbb{Z}_2) \cong q_\star(R(M)) \subseteq \overline{R}(M).$

There also is a regular map $q_\star\co X(M) \to \overline{X}(M),$ and its image is a union of topological components of $\overline{X}(M).$ The fibres over points in the image are again the orbits of the natural action of $H^1(\pi_1(M); \mathbb{Z}_2)$ on $X(M).$ 
Let
$QX(M) = X(M)/H^1(\pi_1(M); \mathbb{Z}_2) \cong q_\star(X(M)) \subseteq \overline{X}(M).$
However, on the level of character varieties, the action may not be free. The points of $X(M)$ fixed by $h \co \pi_1(M)\to \mathbb{Z}_2$ are characterised by $\chi(\gamma)=0$ for all $\gamma \in \pi_1(M)$ with the property that $h(\gamma)=-1.$ This defines an algebraic subset $S(h)$ of $X(M).$ Let $U(M)$ be the Zariski-open set $X(M) \setminus \cup_h S(h).$ Then on $U(M)$ the action of $H^1(\pi_1(M); \mathbb{Z}_2)$ is free and the quotient map $U(M) \to \overline{X}(M)$ onto its image is a regular covering map of its image.

To make these general notions explicit, we appeal to the following material from \cite{Gonzalez-character-1993}.
For each $\gamma\in \pi_1(M)$, there is a regular function $\tau_\gamma \co R(M) \to \mathbb{C}$ defined by $\tau_\gamma(\rho) = \tr \rho(\gamma)$ and a regular function
$I_\gamma\co X(M) \to \mathbb{C}$ defined by $I_\gamma(\chi) = \chi(\gamma).$

Suppose $\gamma_1, \ldots , \gamma_n$ generate $\pi_1(M),$ and consider the $m = \frac{n(n^2+5)}{6}$ functions $\tau_\gamma$ as $\gamma$ ranges over the set 
\[ G = \{\; \gamma_i,\; \gamma_i\gamma_j,\; \gamma_i\gamma_j\gamma_k \;\mid\; 1\le i \le n, 1\le i < j \le n, 1\le i < j < k\le n \;\} \]
Then we obtain a regular map $\tau\co R(M) \to \mathbb{C}^m$ defined by
\[ \rho \mapsto (\tau_\gamma(\rho))_{\gamma \in G}\]
and it is shown in \cite{Gonzalez-character-1993} that there is a natural identification $\tau(R(M)) = X(M).$ We therefore use the notation $\tau\co R(M) \to X(M)$ for the natural map.

Projective characters have a similar interpretation, using the fact that the \emph{square} of the trace of an element in $\PSL(2,\mathbb{C})$ is well defined. Given a representation $\overline{\rho}\co \pi_1(M)\to \PSL(2,\mathbb{C}),$ its character is the map 
\[
\overline{\chi}_{\overline{\rho}} \co \pi_1(M) \to \mathbb{C} \quad \text{ defined by } \quad \gamma \mapsto (\tr \overline{\rho}(\gamma))^2
\] 
As above, let $\overline{\tau}\co \overline{R}(M) \to \overline{X}(M)$ be the natural map from representations to characters.

Throughout this paper, we make use of the following trace identities.
Let $A$,$B$,$C\in \SL(2, \mathbb{C})$. Then
\begin{align}
  \label{tr1}  \tr A^{-1}     =& \tr A \\
  \label{tr2}  \tr (B^{-1}AB) =& \tr A   \\
  \label{tr3}  \tr A\tr B       =& \tr (AB) + \tr (AB^{-1}) \\
  \label{tr4}  \tr (ABC)      =&
  \tr A\tr (BC)+\tr B\tr (AC)+\tr (C)\tr (AB)-\tr A\tr B\tr C-\tr (ACB)
\end{align}

Recall that a representation into $\SL(2,\mathbb{C})$ is \textbf{irreducible} if the only subspaces of $\mathbb{C}^2$ invariant under its image are $\{ 0\}$ and $\mathbb{C}^2.$ Otherwise, a representation is \textbf{reducible}. If a representation is reducible, then it is conjugate into the upper triangular matrices, and hence there is an abelian representation with the same character.
We make frequent use of the following results from \cite{Culler-Shalen-varieties-1983}:

\begin{lemma}
\label{lem:CS irrep lemma}
Suppose $\rho\in R(M).$
\begin{enumerate}
\item  Then $\rho$ is reducible if and only if $\chi_\rho(\gamma)=2$ for each element $\gamma$ of the commutator subgroup of $\pi_1(M).$
\item Suppose $\rho$ is irreducible and $\sigma\in R(M)$ satisfies $\chi_\rho = \chi_\sigma$. Then $\rho$ and $\sigma$ are conjugate and, in particular, $\sigma$ is irreducible.
\end{enumerate}
\end{lemma}

It follows from the above lemma that reducible representations form a closed subset of $R(M)$, and irreducible representations form an open subset of $R(M)$ with respect to the Euclidean topology. Moreover, the subset $X^{\text{red}}(M) \subseteq X(M)$ of all characters of reducible representations is an affine algebraic set, and it is the same as the subset of all characters of abelian representations.

Let $X^{\text{irr}}(M)$ be the union of those Zariski components of $X(M)$ that contain at least one (and hence a Zariski-open) subset of irreducible representations. The following is shown in \cite[Lemma 3.9(iii)]{Porti-torsion-1997}.

\begin{lemma}
\label{lem:porti}
Suppose $\chi \in X^{\text{irr}}(M)\cap X^{\text{red}}(M).$ Then each representation with character $\chi$ is reducible and there is a non-abelian reducible representation with this character.
\end{lemma}

The quotient ${QX}^{\text{irr}}(M) = X^{\text{irr}}(M)/H^1(\pi_1(M),\mathbb{Z}_2 )$ is naturally identified with a union of topological components of the union $\overline{X}^{\text{irr}}(M)$ of Zariski-components of $\overline{X}(M)$ containing irreducible representations. The above lemma implies:

\begin{corollary}
The map $X^{\text{irr}}(M) \to \overline{X}(M)$ is a branched covering map onto its image and its ramification points are characters of irreducible or non-abelian reducible representations in the set $\cup_h S(h).$
\end{corollary}

A point of a complex affine algebraic set $X$ is a \textbf{simple point} if it is contained in a unique Zariski component $X_0$ of $X$ and is a smooth point of $X_0.$ 

Suppose the boundary of $M$ is a union of pairwise disjoint tori and the interior of $M$ admits a complete hyperbolic structure of finite volume. Then the holonomy for the complete hyperbolic structure defines a discrete and faithful $\PSL(2,\mathbb{C})$--character, 
which is unique up to complex conjugation (the ambiguity arises from the two different orientations possible on $M$). Each Zariski component of $\overline{X}(M)$ containing such a character is called a \textbf{canonical component}. 
The discrete and faithful characters lift to $\SL(2,\mathbb{C}),$ and again each Zariski component of $X(M)$ containing at least one of these lifts is called a \textbf{canonical component}. It is shown in \cite[Corollary 3.28]{Porti-torsion-1997} that the discrete and faithful characters are simple points of the canonical components. (The proof in \cite{Porti-torsion-1997} is given in the setting of $\SL(2,\mathbb{C})$--characters, but descends to $\PSL(2,\mathbb{C})$--characters.)

We end this section with a general discussion of the obstruction for lifting characters from $\overline{X}(M)$ to $X(M)$. See \cite[\S4]{Heusener-varieties-2004}, \cite{Culler-lifting-1986}, \cite[\S4.5]{Goldman-deformation-1988}, or \cite[\S2]{Gonzalez-character-1993} for details.
Suppose $\overline{\rho}\in \overline{R}(M).$ Then $\overline{\rho}$ determines a Stiefel-Whitney class $w_2(\overline{\rho})\in H^2(\pi_1(M); \mathbb{Z}_2).$ 
This is zero if and only if $\overline{\rho}$ lifts to a representation $\rho\in R(M).$
Moreover, the class is constant on topological components of $R(M).$
In the case of interest in this paper, $M$ is aspherical and has boundary a single torus.
Hence it is a $K(\pi_1(M), 1)$ and so $H^2(\pi_1(M); \mathbb{Z}_2)\cong H^2(M; \mathbb{Z}_2).$
Moreover, Poincar\'e-Lefschetz duality implies $H^2(M; \mathbb{Z}_2)\cong  H_1(M, \partial M;  \mathbb{Z}_2).$ In particular, if $H_1(M; \mathbb{Z}_2) = \mathbb{Z}_2$, then 
$H^2(\pi_1(M); \mathbb{Z}_2)\cong 0$ and hence every representation lifts.


\subsection{Once-punctured torus bundles}
\label{sec:prelim_bundles}

Throughout the remainder of this paper, $M_\varphi$ denotes a compact core of the hyperbolic once-punctured torus bundle over $S^1$ with monodromy $\varphi$, and $S$ denotes a fixed fibre of $M_\varphi$ with the property that the monodromy is the identity on $\partial S.$ For our fundamental groups, we assume that the base point lies in $\partial S \subset \partial M_\varphi.$ (Note that this notation is a slight variation on the introduction, where we referred to the non-compact manifolds.)

We also denote by $\varphi$ the automorphism $\pi_1(S)\to \pi_1(S)$ induced by the monodromy (this will cause no confusion in this paper). Let $\varphi_* \co H_1(S)\to H_1(S)$ be the induced automorphism on homology. Since $M_\varphi$ is hyperbolic, the associated monodromy is pseudo-Anosov and we have $|\tr \varphi_*|>2.$ 

The fundamental group of $M_\varphi$ admits the presentation 
\begin{equation}\label{eq:fund_gp_HNN}
\pi_1 (M_\varphi) = \langle \; t,a,b\;|\;t^{-1}at=\varphi(a),t^{-1}bt=\varphi(b)\; \rangle
\end{equation}
where $t$ is a meridian of $M_\varphi$ (that is, it corresponds to the trace of the base point of $\pi_1 (M_\varphi)$ under the monodromy) and $a$,$b$ are free generators of $\pi_1 (S)$, chosen such that their commutator is the longitude of $M_\varphi$ (that is, it corresponds to $\partial S$).

We choose the basis of $H_1(S)$ corresponding to the generators $a$,$b$ of $\pi_1(S),$ and often implicitly identify $\varphi_*$ with its corresponding matrix representation, $[\varphi_*]\in \SL(2,\mathbb{Z}).$ 
We may write $\varphi$ as a product of the automorphisms $\alpha$ and $\beta$ induced by the right-handed Dehn-twists about the curves corresponding to $a$ and $b$, since these generate the group of isotopy classes of orientation preserving homeomorphisms of the punctured torus.
The automorphisms are defined as follows (where we also include $\beta^{-1}$ for convenience):
$$
\begin{matrix} 
      \alpha \defeq \begin{cases}
    a\rightarrow a\\
    b\rightarrow ba
    \end{cases}
  &
    \beta \defeq \begin{cases}
    a\rightarrow ab^{-1}\\
    b\rightarrow b
    \end{cases}
      &
    \beta^{-1} = \begin{cases}
    a\rightarrow ab\\
    b\rightarrow b
    \end{cases}
\end{matrix}
$$
This also gives a factorisation of $\varphi_*$ in terms of the induced maps $\alpha_*$ and $\beta_*$ on homology.
We have 
\[
\A \defeq [\alpha_*] = \begin{pmatrix} 1 & 1 \\ 0 & 1 \end{pmatrix}
\qquad\text{and}\qquad
\B \defeq [\beta_*] = \begin{pmatrix} 1 & 0 \\ -1 & 1 \end{pmatrix}
\]

It is a result of Murasugi (see~\cite[Proposition 1.3.3]{Culler-incompressible-1982}) that $M_\varphi$ and $M_\psi$ are homeomorphic if and only if $[\varphi_*]$ is conjugate to either $[\psi_*]$ or  $[\psi_*^{-1}].$ 
The matrices $\A$ and $\B$ generate $\SL(2, \mathbb{Z})$ and satisfy the relations
\[ \A\B\A=\B\A\B, \qquad\text{and}\qquad (\A\B\A)^4=I\]
Using the second relation, we note that any element of $\SL(2, \mathbb{Z})$
can be written as a word in positive powers of $\A$ and $\B$ by substituting
\[ \B^{-1} = \A(\A\B\A)^3\A = \A^2 \B \A^2 \B \A^2 \B \A^2 , \qquad\text{and}\qquad 
\A^{-1} = \B(\B\A\B)^3\B = \B^2 \A \B^2 \A \B^2 \A \B^2
\]
Hence any element of  $\SL(2, \mathbb{Z})$ (including the identity and powers of $\A$ or $\B$) is conjugate to a word of the form
\begin{equation}\label{eq:positive_word}
\A^{a_1}\B^{b_1}\A^{a_2}\B^{b_2} \cdots \A^{a_n}\B^{b_n}
\end{equation}
where $n>0$ and the $a_i$ and $b_i$ are positive integers. Also note that $(\A\B\A)^2=-I.$
The manifold $M_\varphi$ is determined up to homeomorphism by $[\varphi_*]\in \SL(2,\mathbb{Z}).$ 

The choice of a monodromy as a positive word in $\A$ and $\B$ is not unique, and there also is no simple criterion to ensure that the absolute value of the trace is greater than two. 
It is well-known (see \cite[Proposition 2.1]{Gueritaud-canonical-2006} that if $|\tr \varphi_*|>2,$ then the conjugacy class of $\varphi_*$ has a representative of the form 
\begin{equation}\label{eq:positive_standard_word}
\pm \A^{a_1}\B^{-b_1}\A^{a_2}\B^{-b_2} \cdots \A^{a_n}\B^{-b_n}
\end{equation}
where $n>0$, the $a_i$ and $b_i$ are positive integers, and the sign equals the sign of the trace of $\varphi_*.$ This representative is unique up to cyclic permutation of the factors $\A^{a_i}\B^{-b_i}.$ 

If $\tr\varphi_*\neq 2,$ then $H_1(M_\varphi,\mathbb{Z}) \cong \mathbb{Z} \oplus \text{Tor}(M_\varphi,\mathbb{Z}).$ The \emph{parity} of the trace contains information on the rank of homology with coefficients in $\mathbb{Z}_2= \mathbb{Z}/2\mathbb{Z}$. We now describe this in more detail.

The natural homomorphism $\mathbb{Z} \to \mathbb{Z}_2$ gives an epimorphism from $H_1(S)\to H_1(S;\mathbb{Z}_2),$ and takes $\varphi_*$ to a map $\varphi_2\co H_1(S;\mathbb{Z}_2) \to H_1(S;\mathbb{Z}_2).$ On the level of matrices, $[\varphi_*] \mapsto [\varphi_2]$ under the natural epimorphism $\SL(2,\mathbb{Z}) \to \SL(2,\mathbb{Z}_2).$
Since $\SL(2,\mathbb{Z}_2)\cong\Sym(3)$, the order $o(\varphi_2)$ of $\varphi_2$ is either one, two or three. If $\varphi_2$ is not the identity, then it has order two if $\tr (\varphi_2) = 0$ and order three if $\tr (\varphi_2) = 1.$ In particular, $o(\alpha_2)=o(\beta_2)=2$ and the word in \Cref{eq:positive_word} maps to an element of order two if and only if the sum of all exponents is odd.

A direct calculation shows that $b_1(M_\varphi; \mathbb{Z}_2) = \rank H_1(M_\varphi,\mathbb{Z}_2 ) = 4 - o(\varphi_2)\in \{1, 2, 3\}$. To simplify notation, we write ${b_1(\varphi)} = b_1(M_\varphi; \mathbb{Z}_2).$
The maximal rank three is attained if $\varphi_2$ is the identity, and hence for at least one of $\varphi,$ $\varphi^2$ or $\varphi^3.$ Hence each once-punctured torus bundle is either 2-fold or 3--fold covered by  a once-punctured torus bundle with homology with coefficients in $\mathbb{Z}_2$ of maximal rank three.

We identify $H_1(M_\varphi,\mathbb{Z}_2 ) = \Hom(\pi_1(M_\varphi), \mathbb{Z}_2) = H^1(M_\varphi,\mathbb{Z}_2),$ and denote $\varepsilon\co \pi_1(M_\varphi) \to \mathbb{Z}_2$ the homomorphism defined by
\[ \varepsilon(t) = -1, \qquad \varepsilon(a) = 1, \qquad \varepsilon(b) = 1\]
Let $H \le \Hom(\pi_1(M_\varphi), \mathbb{Z}_2)$ be the subgroup consisting of all homomorphisms $h \co \pi_1(M_\varphi) \to \mathbb{Z}_2$ satisfying $h(t)=1.$ Note that $H$ has rank zero if $\tr (\varphi_*)$ is odd. In the following, we will make use of the splitting $H^1(M_\varphi,\mathbb{Z}_2 ) = \langle \varepsilon \rangle \oplus H.$


\subsection{Mutation}
\label{sec:mutation}

The once-puctured torus $S$ admits an orientation preserving involution $\iota \co S \to S$. If one cuts $M_{\varphi}$ along $S$ and changes the monodromy by $\iota$, one obtains $M_{\iota\varphi}.$ This process is called \textbf{mutation}. We have $\iota_* = -I$ and 
$$
\begin{matrix} 
    \iota = \begin{cases}
    \;a\rightarrow ba^{-1}b^{-1}\\
    \;b\rightarrow bab^{-1}a^{-1}b^{-1}
    \end{cases}
\end{matrix}
$$
It follows from \cite[Proposition 1]{Tillus-character-2004} that there is a birational equivalence between Zariski components of $\overline{X}(M_{\varphi})$ and $\overline{X}(M_{\iota\varphi})$ that contain the character of a representation $\overline{\rho}$ whose restriction to $S$ is irreducible and has trivial centraliser. In particular, if one is interested in the topology of generic components of $\overline{X}(M_{\varphi}),$ then one may restrict attention to bundles where the trace of $\varphi_*$ is positive. Moreover, components where all irreducible representations have non-trivial centraliser can be understood directly (see \Cref{sec:PSLorigin_axes}).

In the case of $\SL(2, \mathbb{C})$-character varieties, using the HNN-splitting for the fundamental group as in the proof of \cite[Proposition 1]{Tillus-character-2004}, one sees that the restriction maps to the fibre result in the same fixed-point set; that is: $X_\varphi(S)=X_{\iota\varphi}(S).$ However, the topology of the $\SL(2, \mathbb{C})$-character varieties may be different. For instance, mutation of the figure eight knot complement (with $\varphi_*=\A\B^{-1}$) along the fibre results in the associated sister manifold, and the smooth projective models of their canonical components are a torus and a sphere respectively \cite[\S3.4]{Tillus-character-2004}.


\subsection{The character variety of the free group of rank two}
\label{sec:charF2}

The fundamental group of the once-punctured torus fibre $S$ is the free group 
$\langle a,b\rangle$, and the character variety $X(S)$ admits the affine coordinate $(x,y,z)\in\mathbb{C}^3,$ where $$x=\tau_a(\rho)\text{, }y=\tau_b(\rho)\text{, }z=\tau_{ab}(\rho)$$ as $\rho$ ranges over $R(S)$. An elementary calculation, which goes back to classical work of Fricke and Klein, shows that $X(S) = \mathbb{C}^3.$

Let $\rho\in R(S)$ be an irreducible representation. Up to conjugation, $\rho$ is of the form:
\begin{equation}
\label{eq:standard_irrep}
\begin{matrix}
  \rho(a) = \begin{pmatrix}s&0\\1&s^{-1}\end{pmatrix}  & \rho(b) = \begin{pmatrix}p&u\\0&p^{-1}\end{pmatrix}
\end{matrix}
\end{equation}
where $u\neq 0$.

If $\rho\in R(S)$ is a reducible representation, then $\rho$ is conjugate to one of the following:
\begin{equation}\label{eq:I}
\begin{matrix}
  \rho(a) = \begin{pmatrix}s&0\\0&s^{-1}\end{pmatrix}  & \rho(b) = \begin{pmatrix}p&u\\0&p^{-1}\end{pmatrix}
\end{matrix}
\end{equation}
\begin{equation}\label{eq:II}
\begin{matrix}
 \rho(a) = \begin{pmatrix}\pm 1&1\\0&\pm 1\end{pmatrix}  & \rho(b) = \begin{pmatrix}p&u\\0&p^{-1}\end{pmatrix}
\end{matrix}
\end{equation}
Every point in $X^{\text{red}}(S)$ is the character of a reducible character as in \Cref{eq:I} with $u=0$.
A representation $\rho\in R(S)$ is reducible if and only if its character satisfies the equation $$x^2+y^2+z^2-xyz-4=0$$
This is a consequence of the first part of \Cref{lem:CS irrep lemma}. See \cite{Culler-Shalen-varieties-1983} for details.

A parameterisation of the $\PSL(2, \mathbb{C})$--character variety $\overline{X}(S)$ is described in \cite[\S4.2]{Heusener-varieties-2004} as follows. Since the fundamental group of $S$ is free in two generators, every character in $\overline{X}(S)$ lifts to $X(S)$. By studying the invariant functions on $X(S),$ one obtains an identification
\begin{equation}\label{eq:param_overline{X}(S)}
\overline{X}(S) \cong \{ (X,Y,Z,W) \in \mathbb{C}^4 \mid W^2 = XYZ \}
\end{equation}
where $X = \tau_a^2,$ $Y = \tau_b^2,$ $Z= \tau_{ab}^2$ and 
$W = \tau_a\tau_b\tau_{ab} = \frac{1}{2} (\tau_a^2\tau_b^2 + \tau^2_{ab} - \tau^2_{ab^{-1}})$
The four-fold branched covering map 
$X(S)\to \overline{X}(S)$ thus has the following description in these coordinates:
\begin{equation}\label{eq:param_overline{X}(S)_quotient}
\overline{X}(S) \ni (x, y, z) \mapsto (x^2, y^2, z^2, xyz) \in \overline{X}(S)
\end{equation}


\subsection{Thurston's dimension bound}
\label{app:thurston}

Based on Thurston~\cite[Theorem 5.6]{Thurston-notes}, Culler and Shalen~\cite[Proposition 3.2.1]{Culler-Shalen-varieties-1983}) provided the following result:

\begin{proposition}[Thurston]
\label{pro:dimSL_Thurston}
Let $N$ be a compact orientable 3-manifold.
Let $\rho \co \pi_1(N)  \to \SL(2, \mathbb{C})$ be an irreducible representation such that for each torus component $T$ of $\partial N$, $\rho(\im( \pi_1(T) \to \pi_1(N))) \not\subseteq \{ \pm 1\}.$ Let $R_0$ be an irreducible component of $R(N)$ containing $\rho$. Then $X_0 = \tau(R_0)$ has dimension $\ge s - 3\chi(N),$ where $s$ is the number of torus components of $\partial N.$
\end{proposition}

An analogous statement for representations into $\PSL(2, \mathbb{C})$ is the following:

\begin{proposition}[Thurston]
\label{pro:dimPSL_Thurston}
Let $N$ be a compact orientable 3-manifold.
Let $\prho_0 \co \pi_1(N)  \to \PSL(2,\mathbb{C})$ be an irreducible representation such that for each torus component $T$ of $\partial N$, $\prho_0(\im( \pi_1(T) \to \pi_1(N)))$ is non-trivial and not isomorphic with $\mathbb{Z}_2 \oplus \mathbb{Z}_2.$ Let $\overline{R}_0$ be an irreducible component of $\overline{R}(N)$ containing $\overline{\rho}$. Then $\overline{X}_0 = \overline{\tau}(\overline{R}_0)$ has dimension $\ge s - 3\chi(N),$ where $s$ is the number of torus components of $\partial N.$
\end{proposition}

The group $\mathbb{Z}_2 \oplus \mathbb{Z}_2$ does not occur as an exclusion in the statement of \cite[Theorem 5.6]{Thurston-notes} because Thurston was concerned with \emph{holonomies of hyperbolic structures} and not arbitrary irreducible representations. Its appearance is rather natural in light of the following observations.

First, abelian subgroups of $\PSL(2,\mathbb{C})$ either have a global fixed point in $P^1(\mathbb{C})$, or they are isomorphic to the Klein four group 
$\mathbb{Z}_2 \oplus \mathbb{Z}_2$ (and are irreducible but elementary). Up to conjugacy, the latter has the non-trivial elements\footnote{Following standard convention, we write elements of $\PSL(2,\mathbb{C})$ not as sets $\{\pm C\}$ but rather in the form $\pm C$.}
\begin{equation}
    \kappa_1 = \pm\begin{pmatrix}
    0&1\\-1&0
    \end{pmatrix} 
    \qquad
    \kappa_2 = \pm\begin{pmatrix}
    0&i\\i&0
    \end{pmatrix}
     \qquad\
    \kappa_3 = \pm\begin{pmatrix}
    i&0\\0&-i
    \end{pmatrix}
\end{equation}
In particular, for representations of 
$\langle \mu, \lambda \rangle \cong \mathbb{Z} \oplus \mathbb{Z},$ the two cases are distinguished by $\tr\rho[\mu, \lambda] =2$ and $\tr\rho[\mu, \lambda] =-2,$ with the first equation determining a positive dimensional component in the $\PSL(2,\mathbb{C})$--character variety of $\mathbb{Z} \oplus \mathbb{Z}$ and the second an isolated point.
It also follows that an equivalent requirement on $\prho_0$ in the statement of \Cref{pro:dimPSL_Thurston} is that 
\begin{center}
the image of each peripheral torus subgroup is non-trivial and reducible.
\end{center}

Second, each lift of the Klein four group to $\SL(2,\mathbb{C})$ is isomorphic to the quaternionic group and hence if a peripheral subgroup corresponding to a torus boundary component has this image under a projective representation, then it does not lift. This explains why this case does not appear in \Cref{pro:dimSL_Thurston}.

\begin{proof}[Proof of \Cref{pro:dimPSL_Thurston}]
We follow the wording of the proof of \cite[Proposition 3.2.1]{Culler-Shalen-varieties-1983} closely, and combine this with observations from Thurston's proof and results stated in \cite{Heusener-varieties-2004}.

By \cite[Corollary 3.3.5]{Heusener-varieties-2004}, the conclusion is equivalent to the assertion that 
$\dim \overline{R}_0 \ge s - 3\chi(N) + 3.$ We shall prove this by induction on $s.$ First suppose that $s=0.$ We may assume that $\partial N \neq \emptyset$ since otherwise $\chi(N)=0$ and there is nothing to prove. Now $N$ has the homotopy type of a finite 2--dimensional CW-complex $K$ with one 0--cell, and, say, $m$ 1--cells and $n$ 2--cells. Thus $\pi_1(N)$ has a presentation 
$\langle g_1, \ldots, g_m \mid r_1 = \ldots = r_n =1\rangle.$
This presentation gives a natural identification of $\overline{R}(N)$ with an algebraic subset of $(\PSL(2,\mathbb{C}))^m.$

Suppose $\prho_0(g_i) = \pm G_i.$ Then for each $r_j$ there exists 
$\epsilon_j \in \{\pm 1\}$ such that 
\[ 
r_j(G_1, \ldots, G_m) = \epsilon_j \begin{pmatrix} 1 & 0 \\ 0 & 1 \end{pmatrix}
\]
Note that if $\epsilon_j =1$ for each $j,$ then $\prho_0$ lifts to a representation to $\SL(2,\mathbb{C}).$ Now consider the coordinate ring 
$\mathbb{C}[a_1, b_1, c_1, d_1, \ldots, a_m, b_m, c_m, d_m].$ Let $\overline{R}_0$ be the variety with ideal defined by the $m$ equations $a_ic_i - b_id_i=1$ and the $3n$ equations arising from the $(1,1), (1,2)$ and $(2,1)$ entries of the $m$ matrix equations
\[ 
r_j\bigg( \begin{pmatrix} a_1 & b_1 \\ c_1 & d_1 \end{pmatrix}, \ldots, \begin{pmatrix} a_m & b_m \\ c_m & d_m \end{pmatrix}\bigg) = \epsilon_j \begin{pmatrix} 1 & 0 \\ 0 & 1 \end{pmatrix}
\]
The equations given by the $(2,2)$ entries of the above matrix equations are a consequence of the other equations and the multiplicative property of determinants. Our set-up gives a natural algebraic embedding $\overline{R}_0 \to (\PSL(2,\mathbb{C}))^m$ with image in $\overline{R}(N).$ Hence we may identify $\overline{R}_0$ with a subvariety of $\overline{R}(N)$ containing $\prho_0.$ Now 
\[
\dim \overline{R}_0 \ge 4m-m-3n = 3m-3n = -3\chi(K)+3 = -3\chi(N)+3
\]
gives the desired conclusion.

Now suppose $s>0.$ Let $T$ be a torus component of $\partial N.$ Since 
$\prho_0(\im( \pi_1(T) \to \pi_1(N))) \not\subseteq \{ 1\},$ there is $\alpha \in \pi_1(N)$, represented by a simple closed curve on $T$, such that $\prho_0(\alpha) \neq 1 \in \PSL(2,\mathbb{C}).$ Since $\prho_0$ is irreducible, there is $\gamma \in \pi_1(N)$ such that $\prho_0$ restricted to the subgroup generated by $\alpha$ and $\gamma$ is irreducible. 

Following Thurston, we show that $\gamma$ can be chosen such that $\tr^2\prho_0(\gamma) \neq 4.$ Hence suppose that $\tr^2\prho_0(\gamma) = 4.$ 
Since $\prho_0(\alpha)$ and $\prho_0(\gamma)$ have no common fixed point on $P^1(\mathbb{C}),$ it follows that $\prho_0(\gamma)$ is parabolic and has a unique fixed point on $P^1(\mathbb{C}).$
We may conjugate $\prho_0$ such that 
\[ 
\prho_0(\gamma) =  \pm \begin{pmatrix} 1 & 1 \\ 0 & 1\end{pmatrix}
\qquad \text{and} \qquad
\prho_0(\alpha) =  \pm \begin{pmatrix} a & b \\ c & d\end{pmatrix}
\]
where $c \neq 0.$ Then
\[ 
\prho_0(\alpha\gamma^{2n}) =  \pm \begin{pmatrix} a & b+2na \\ c & d+2nc\end{pmatrix}
\]
Hence $\tr^2\prho_0(\alpha\gamma^{2n}) = (a+d+2nc)^2$ and since $c\neq 0$ we may choose $n$ such that $\tr^2\prho_0(\alpha\gamma^{2n})$ does not equal 4. Now the parabolic $\prho_0(\gamma^{2n})$ has the same fixed point on $P^1(\mathbb{C})$ as 
$\prho(\gamma).$ Hence $\prho_0(\alpha)$ and $\prho_0(\gamma^{2n})$ have no common fixed point on $P^1(\mathbb{C}),$ and so the representation restricted to 
$\langle \alpha , \gamma^{2n} \rangle = \langle \alpha , \alpha\gamma^{2n} \rangle$ is irreducible. This completes the argument that we may choose $\gamma$ such that $\tr^2\prho_0(\gamma) \neq 4.$

Choose a base point for the fundamental group of $N$ based on $T$ and thus represent $\gamma$ by a simple closed curve based at this point. Drilling out a regular open neighbourhood of $\gamma$ gives a submanifold $M$ of $N$ such that $N$ is obtained by adding a 2--handle to a genus two boundary component $S$ of $M.$ We may choose a standard basis $\alpha', \beta', \gamma', \delta'$ of $\pi_1(S)$ such that $\alpha'$ and $ \gamma'$ are mapped to $\alpha$ and $ \gamma$ under the natural surjection $i_*\co \pi_1(M) \to \pi_1(N)$; the 2--handle is attached to $S$ along a simple closed curve that represents the conjugacy class of $\delta;$ and the relation for $\pi_1(S)$ is given by $[\alpha', \beta'] = [\gamma', \delta'].$

The representation $\prho'_0 = \prho_0\circ i_* \co \pi_1(M)\to \PSL(2,\mathbb{C})$ is irreducible since $i_*$ is a surjection. By the induction hypothesis, there is an irreducible component $\overline{R}'_0$ of $\overline{R}(M)$ that contains $\prho'_0$ and has dimension
\[
\dim \overline{R}'_0 \ge (s-1) - 3 \chi(M) + 3 = s - 3 \chi(N) + 5
\]
since $\chi(M) = \chi(N)-1.$ Let $W$ be the intersection of $\overline{R}'_0$ with the subvariety defined by the two equations
\begin{align}
\tr^2\prho (\delta') &= 4\\
\tr\prho[\gamma', \delta'] &= 2
\end{align}
By  \cite[\S2.4]{Heusener-varieties-2004}, these are well defined invariant functions on $\overline{R}'_0.$ Then the dimension of each component of $W$ is at least $s - 3 \chi(N) + 3.$

We first show that $\prho'_0$ is contained in $W.$
The kernel of $i_*$ is the normal closure of $\delta'$, and hence $\prho'_0(\delta') = 1.$ This implies 
$\tr^2\prho'_0(\delta') = 4$ and $\tr\prho'_0[\gamma', \delta'] = 2.$ Hence there is an irreducible component $\overline{R}_0$ of $W$ that contains $\prho'_0.$
To obtain the desired conclusion that we can identify $\overline{R}_0$ with a subvariety of $\overline{R}(N)$ it remains to show that 
all representations $\prho \in \overline{R}_0$ near $\prho_0$ satisfy $\prho(\delta') = 1.$ 
 
Hence suppose that for $\prho \in \overline{R}_0$ near $\prho_0,$ we have $\prho(\delta') \neq 1.$ Then $\prho(\delta')$ is a non-trivial parabolic. Since 
$\tr\prho[\gamma', \delta'] = 2,$ 
it follows that $\prho(\delta')$ and $\prho(\gamma')$ share a fixed point on $P^1(\mathbb{C}).$  Since $\prho$ is near $\prho_0,$ $\prho(\gamma')$ is not a parabolic. Hence $\prho[\gamma', \delta']$ is a non-trivial parabolic that has the same fixed point as $\prho(\delta')$ and shares a fixed point with $\prho(\gamma').$

Since $[\gamma', \delta'] = [\alpha', \beta']$ in $\pi_1(M),$ we have $\tr\prho[\alpha', \beta'] = \pm 2.$ Since $\prho$ is near $\prho_0,$ it follows from our hypothesis on the peripheral subgroups that $\tr\prho[\alpha', \beta'] = 2.$ Hence $\prho(\alpha')$ and $\prho(\beta')$ have a common fixed point on $P^1(\mathbb{C}).$ Hence the unique fixed point of the non-trivial parabolic $\prho[\gamma', \delta'] = \prho[\alpha', \beta']$ is in common with both $\prho(\gamma')$ and $\prho(\alpha')$ and thus $\tr\prho[\alpha', \gamma'] = 2.$ But this is not possible since $\prho$ is near $\prho'_0$, and $\prho'_0$ restricted to the subgroup generated by $\alpha'$ and $\gamma'$ is irreducible. This is the final contradiction to our hypothesis that $\prho(\delta') \neq 1.$
\end{proof}

\begin{example}
The $\PSL(2,\mathbb{C})$--character variety of the Whitehead link complement contains 0--dimensional components that satisfy the condition that all peripheral subgroups have image isomorphic with $\mathbb{Z}_2 \oplus \mathbb{Z}_2.$ See \cite[Appendix C]{Tillus-tropical-2020}.
\end{example}

\begin{example}
\label{exa:BP_dim0is1}
Baker and Petersen~\cite[Theorem 7.6]{baker-character-2013} give zero-dimensional components in character varieties of once-punctured torus bundles that appear to contradict \Cref{pro:dimPSL_Thurston}. A less sophisticated approach than given in \cite{baker-character-2013} shows that these characters indeed lie on one-dimensional components. The presentation for the fundamental group of $M_n$ with monodromy $\varphi_n = \A\B^{n+2}$
is given in \cite[\S2.1]{baker-character-2013} as
\[
\Gamma_n = \langle \alpha, \beta \mid \beta^{-n} = \alpha^{-1}\beta\alpha^2\beta\alpha^{-1}\rangle
\]
where a meridian is the element $\mu = \beta\alpha$ and a corresponding longitude is
$\lambda = (\alpha \beta \alpha^{-1})\beta (\alpha \beta \alpha^{-1})^{-1}\beta^{-1}$
(see \cite[Lemma 2.3]{baker-character-2013}). Suppose  $|n|\ge 2$ is even. Then
\[
\prho(\alpha) 
= \pm 
\begin{pmatrix} 
	-y x^n &  -1 \\
	\frac{ 1 + x^{2 + n} }{ (1 - x^2) (1 - x^n) } &  y \end{pmatrix}
\qquad \text{and} \qquad
\prho(\beta) = \pm\begin{pmatrix} x & 0 \\ 0 & x^{-1} \end{pmatrix}
\]
gives a representation of $\Gamma_n$ into $\PSL(2,\mathbb{C})$ that does not lift to $\SL(2,\mathbb{C})$ if
\[
 x^2 + x^n = y^2 x^n(1-x^2)(1-x^n)
\]
Indeed, the above equation is equivalent to the determinant condition for $\prho(\alpha)$, and it is also a sufficient condition for $\prho(\alpha)\prho(\beta)^{-n}\prho(\alpha) + \prho(\beta)\prho(\alpha)^2\prho(\beta)=0$ (noting that both summands are independent of the choice of sign).

The alleged isolated points in \cite[Theorem 7.6]{baker-character-2013}  correspond to the points $(x_0, y_0)$ satisfying $x_0^n=-1$ and $y_0 = \pm \frac{1}{\sqrt{2}}$ on this curve. At each of these points, the image of the commutator of meridian and longitude has trace equal to two (hence the representation restricted to the peripheral subgroup is reducible) and the image of the meridian has trace $\pm \frac{1}{\sqrt{2}}(x_0+x_0^{-1})$ (hence the representation restricted to the peripheral subgroup is non-trivial since the absolute value of the trace is bounded from above by $\sqrt{2}$).
The triple of squares of traces is indeed as stated in  \cite[Theorem 7.6]{baker-character-2013}
\[
(\;\tr^2\prho(\alpha),\; \tr^2\prho(\beta),\; \tr^2\prho(\alpha\beta)\;) = (\;2,\; (x_0+x_0^{-1})^2,\;\frac{1}{2}(x_0+x_0^{-1})^2\;)
\]
We have intentionally used the notation from \cite{baker-character-2013} in this example. With respect to the notation of \Cref{eq:fund_gp_HNN}, one has $t= \mu^{-1}$, $a = \beta^{-1}\mu\beta\mu^{-1}$ and $b = \beta$ and
\begin{equation*}
\pi_1 (M_n) = \langle \; t,a,b\;|\;t^{-1}at=a(a^{-1}b^{-1})^{n+2},\;t^{-1}bt=ba\; \rangle
\end{equation*}
\end{example}


\section{Character varieties of once-punctured torus bundles}
\label{sec: char of once-punctured torus bundles}

We continue with the notation introduced in \Cref{sec:Preliminaries} and analyse the $\SL(2, \mathbb{C})$--character varieties of once-punctured torus bundles using the restriction to the fibre.


\subsection{The restriction map}
\label{sec:Restriction}
The fundamental group of a once-punctured torus bundle $M_\varphi$ with monodromy $\varphi$ is $$\pi_1 (M_\varphi) = \langle t,a,b|t^{-1}at=\varphi(a),t^{-1}bt=\varphi(b) \rangle$$

Throughout, we assume $[\varphi_*] = \begin{pmatrix}k_1&k_2\\k_3&k_4\end{pmatrix}\in \SL(2,\mathbb{Z})$. Since we restrict our attention to hyperbolic once-punctured torus bundles, we have $|\tr \varphi_*| = |k_1+k_4|>2$.

As in \Cref{sec:basic_char}, $X(M_\varphi)$ admits the affine coordinate $(x,y,z,u,v,w,q)$ in $\mathbb{C}^{7}$ where 
$$x=\tau_a(\rho)\text{, }y=\tau_b(\rho)\text{, }z=\tau_{ab}(\rho)\text{, }u=\tau_{t}(\rho)\text{, }v=\tau_{at}(\rho)\text{, }w=\tau_{bt}(\rho)\text{, }q=\tau_{abt}(\rho)$$ for $\rho\in R(M_\varphi)$.
Consider the restriction map
\begin{align*}
    r: X(M_\varphi) &\to X(S) \\
    (x,y,z,u,v,w,q) &\mapsto (x,y,z)
\end{align*}
Each representation $\rho\in R(M_\varphi)$ satisfies
 \begin{equation*}
    \rho(t)^{-1}\rho(\gamma)\rho(t) = \rho(\varphi(\gamma))\text{, }\forall \gamma\in \pi_1 (S)
\end{equation*}
Taking the traces on both sides, 
$$\tr(\rho(\gamma)) = \tr(\rho(\varphi(\gamma))\text{, }\forall \gamma\in \pi_1 (S)$$ 

Since $\chi_\rho$ in $X(S)$ is uniquely determined by $\tr\,\rho(a)$, $\tr\,\rho(b)$ and $\tr\,\rho(ab)$, the image of $r$ satisfies
\begin{align*}
    \im\,(r) \subseteq X_\varphi (S) \defeq \{(\tr\rho(a),\tr\rho(b),\tr\rho(ab))|\tr\,\rho(a)&=\tr\,\rho(\varphi_*(a)),\\ \tr\rho(b) &= \tr\rho(\varphi_*(b)),\\ \tr(\rho(ab))&=\tr\rho(\varphi_*(ab)),\, \rho\in R(S)\} \subseteq X(S)
\end{align*}
Since $\tr\rho(\varphi_*(a))$, $\tr\rho(\varphi_*(b))$ and $\tr\rho(\varphi_*(ab))$ can be written as polynomials with rational coefficients in $x=\tr\rho(a)$, $y=\tr\rho(b)$ and $z=\tr\rho(ab)$, and since $\varphi_*$ is an automorphism of the free group in two letters, it follows that there is a polynomial automorphism $\overline{\varphi}\co \mathbb{C}^3 \to \mathbb{C}^3$ induced by $\varphi$ with the property that $X_\varphi (S)$ is the set of fixed points of this polynomial automorphism:
\begin{equation*}
   X_\varphi (S) = \{(x,y,z) \in \mathbb{C}^3 |(x,y,z)=\overline{\varphi}(x,y,z) \}
\end{equation*}

\begin{lemma}\label{lem:im(r)}
We have $\im (r) = X_\varphi (S).$ Moreover, there are at most finitely many reducible characters in $X_\varphi (S).$
\end{lemma}

\begin{proof}
By definition, $\im (r) \subseteq X_\varphi (S)$ and for the first statement, it remains to show that the map is surjective.

If $(x,y,z) \in X_\varphi (S)$ is the character of an irreducible representation, then $(x,y,z) \in\im (r)$
follows from the presentation of the group and the second statement in \Cref{lem:CS irrep lemma}. We analyse this situation in more detail in  \Cref{lem:irreps in fixed point set}.

Hence suppose $(x,y,z) \in X_\varphi (S)$ is the character of a reducible representation. Since this is also the character of an abelian representation, we may assume that it is the character of the representation $\rho$ defined by
\begin{equation*}
    \rho(a) = \begin{pmatrix}
    s&0\\0&s^{-1}
    \end{pmatrix}\qquad \rho(b) = \begin{pmatrix}
    p&0\\0&p^{-1}
    \end{pmatrix}
\end{equation*}
Since $\rho(\varphi_*(a))$ can be viewed as a word in these matrices, it is also diagonal. Similarly for $\rho(\varphi_*(b)).$ Hence
\begin{equation*}
  \rho(\varphi_*(a)) = \begin{pmatrix}s_1&0\\0&s_1^{-1}\end{pmatrix}  \qquad \rho(\varphi_*(b)) = \begin{pmatrix}p_1&0\\0&p_1^{-1}\end{pmatrix}
\end{equation*}
Then the three trace conditions defining $X_\varphi (S)$ imply that we have the mutually exclusive cases:
\begin{enumerate}
\item[(1)] $(s_1, p_1) = (s,p)$ or 
\item[(2)] $(s_1, p_1) = (s^{-1}, p^{-1})$ with either $s \neq s^{-1}$ or $p \neq p^{-1}$ (equivalently, $s\neq \pm 1$ or $p\neq \pm 1)$
\end{enumerate}
In order to show that $(x,y,z) \in \im (r),$ we need to determine  $T \in \SL(2,\mathbb{C})$ such that $T^{-1} \rho T = \rho\varphi_*.$ 

In the first case, this is satisfied by $T = \begin{pmatrix}m&0\\0&m^{-1}\end{pmatrix}$ for arbitrary $m \in \mathbb{C}^*.$ We note that we have $\rho(\varphi_*(\gamma)) = \rho(\gamma)$ for all $\gamma\in \pi_1(S),$ and
the resulting representation of $\pi_1(M_\varphi)$ is abelian and hence reducible.

In the second case, this is satisfied by $T = \begin{pmatrix}0&-1\\1&0\end{pmatrix}.$ 
We note that we have  $\rho(\varphi_*(\gamma)) = \rho(\gamma)^{-1}$ for all $\gamma\in \pi_1(S),$ and the resulting representation of $\pi_1(M_\varphi)$ is irreducible and corresponds to a \textbf{binary dihedral representation} in $R(M_\varphi(\lambda))\subset R(M_\varphi)$ (see \cite[Section 4]{Boyer-algebraic-2002}).

Hence in either case, this shows that $(x,y,z) \in \im (r).$

We now prove the second statement. 
The monodromy $\varphi_* = \begin{pmatrix}k_1&k_2\\k_3&k_4\end{pmatrix}\in SL(2,\mathbb{Z})$ satisfies $| k_1+k_4 |>2$. The pair of eigenvalues of the abelian representation $\rho$ as above satisfies $(s^{k_1}p^{k_3}, s^{k_2}p^{k_4}) = (s^{\pm 1},p^{\pm 1})$. This implies that each of $s$ and $p$ raised to the power of $(k_1\pm 1)(k_4\pm 1) - k_2k_3$ equals one. There are finitely many solutions unless 
$(k_1\pm 1)(k_4\pm 1) = k_2k_3.$ Using the determinant condition, this is equivalent with $k_1+k_4 = \mp 2,$ which contradicts the trace condition for $\varphi_*.$
This proves the second statement.
\end{proof}

\begin{scholion}
\label{schol:binary_dihedral}
The irreducible characters in $X(M_\varphi)$ that map to reducible characters in $X_\varphi (S)$ are precisely the characters of the binary dihedral characters in $X(M_\varphi(\lambda))\subset X(M_\varphi).$
\end{scholion}

 
\subsection{The origin and the coordinate axes}
\label{subsec:lines}

For $\alpha$ and $\beta^{-1}$, the induced polynomial automorphisms are 
\[
\overline{\alpha}(x,y,z) = (\;x,\;z,\;xz-y\;) \quad\text{and}\quad \overline{\beta^{-1}}(x,y,z) = (\;z,\;y,\;yz-x\;)
\]
respectively. Given the standard representative $\varphi_*$ as in \Cref{eq:positive_standard_word}, the associated polynomial automorphism $\overline{\varphi}\co \mathbb{C}^3 \to \mathbb{C}^3$ induced by $\varphi$ factors in terms of $\overline{\iota},$
$\overline{\alpha}$ and $\overline{\beta^{-1}}.$

This implies that the origin $(0,0,0) \in X(S)$ is always contained in $X_\varphi(S)$ for any $M_\varphi.$ Restricted to $S,$ any representation with this character has image in $\SL(2, \mathbb{C})$ isomorphic to the quaternionic group, and image isomorphic to the Klein four group under the quotient map to $\PSL(2, \mathbb{C})$ (see \Cref{eq:paramK4} below).

The examples in \Cref{sec:second_family} show that the component containing $(0,0,0)$ in $X_\varphi(S)$ may have arbitrarily large genus if $b_1(\varphi) \neq 3.$ We now show that in the maximal rank case, at least one of the coordinate axes is always contained in $X_\varphi(S).$

For all $x, y, z \in \mathbb{C}$, we have $\overline{\iota}(x, y, z) = (x, y, z)$ and
\begin{equation*}
    \begin{matrix}
      \begin{cases}
      \overline{\alpha}(x,0,0) = (x,0,0)\\
      \overline{\alpha}(0,y,0) = (0,0,-y)\\
      \overline{\alpha}(0,0,z) = (0,z,0)
      \end{cases}    & \begin{cases}
      \overline{\beta^{-1}}(x,0,0) = (0,0,-x)\\
      \overline{\beta^{-1}}(0,y,0) = (0,y,0)\\
      \overline{\beta^{-1}}(0,0,z) = (z,0,0)
      \end{cases} 
    \end{matrix}
\end{equation*}
Denote the $x$, $y$, and $z$ coordinate axes in $\mathbb{C}^3$ by $L_1,$ $L_2$ and $L_3$ respectively. Then $\overline{\alpha}$ and $\overline{\beta^{-1}}$ permute these three lines, with a sign change for exactly one of them. Indeed, we may assume that the permutation of the lines corresponds to the permutation of the indices of the lines induced by $[\alpha_2]$ and $[\beta_2]$ via the isomorphism $\SL(2,\mathbb{Z}_2)\cong \Sym(3)$.

\begin{lemma}
\label{lem:dimension_bound_order1}
If $H_1(M_\varphi,\mathbb{Z}_2)\cong \mathbb{Z}_2^3$, then $X_\varphi (S)$ contains either exactly one or all three of $L_1$,$L_2$ and $L_3$. Moreover, $X_{\varphi^2} (S)$ contains all three lines. 
\end{lemma}
\begin{proof}
Since $\overline{\iota}(x, y, z) = (x, y, z),$ it suffices to restrict to the case where $\tr \varphi_*>0.$
We may assume that $\varphi_*$ is given as in \Cref{eq:positive_standard_word}. 
Note that $H_1(M_\varphi,\mathbb{Z}_2)\cong \mathbb{Z}_2^3$ is equivalent to $o(\varphi_2)=1$. Then 
\[\text{Id} = [\varphi_2] = [\alpha_2]^{a_1}[\beta_2]^{-b_1}[\alpha_2]^{a_2}[\beta_2]^{-b_2}\cdots[\alpha_2]^{a_n}[\beta_2]^{-b_n}\]
It follows from the description of the action of $\alpha_2,\beta_2$ on $L_1$,$L_2$ and $L_3$, that the induced polynomial automorphism $\overline{\varphi}$ stabilises each line. Hence, $\overline{\varphi}(L_i) = L_i$. Note that when $\overline{\alpha}$ (resp.\thinspace $\overline{\beta^{-1}}$) permutes the lines, the sign of exactly one line changes. Denote the number of sign changes of $L_i$ by $n_i$. We have
\[\begin{cases}
      \overline{\varphi}(x,0,0) = ((-1)^{n_1}x,0,0)\\
      \overline{\varphi}(0,y,0) = (0,(-1)^{n_2}y,0)\\
      \overline{\varphi}(0,0,z) = (0,0,(-1)^{n_3}z)
      \end{cases} \]
As $o(\varphi_2) = 1$, the number $n_1+n_2+n_3 = \sum_{i=1}^n (a_i+b_i)$ is also even. Hence at least one $n_i$ is even and at least one coordinate is fixed by $\overline{\varphi}$. 
We also see that either exactly one $n_i$ is even or all three $n_i$ are even.
This proves the first statement. The second statement follows from the observation that
\[\begin{cases}
      \overline{\varphi}^2(x,0,0) = ((-1)^{2n_1}x,0,0) = (x, 0,0)\\
      \overline{\varphi}^2(0,y,0) = (0,(-1)^{2n_2}y,0)=(0,y,0)\\
      \overline{\varphi}^2(0,0,z) = (0,0,(-1)^{2n_3}z)=)0,0,z)
      \end{cases} \]
This completes the proof of the lemma.      
\end{proof}
If $b_1(M_\varphi) = 1$, then $[\varphi_2]$ has order $3.$ So $\overline{\varphi}$ does not stabilise any axis and $L_i\cap X_\varphi(S) = \{(0,0,0)\}$ for each $i.$

We also note that for any $\varphi,$ $X_{\varphi^6} (S)$ contains all three lines.

As a last observation, we add that $(x,0,0) \in L_1$ is a reducible character if and only if $x^2=4.$ In this case, the extension to $\pi_1(M_\varphi)$ is either reducible or binary dihedral according to \Cref{schol:binary_dihedral}. 


\subsection{Dimensions and cyclic covers}

\begin{proposition}
\label{pro:dimension_bound}
Let $M_\varphi$ be a hyperbolic once-punctured torus bundle. Then every Zariski component of $X(M_\varphi)$ is one-dimensional.
\end{proposition}

\begin{proof}
We first note that since $H_1(M_\varphi,\mathbb{Z}) \cong \mathbb{Z} \oplus \text{Tor}(M_\varphi,\mathbb{Z})$, every component containing only reducible characters is one-dimensional and of genus zero.

Suppose $X\subset X(M_\varphi)$ is a component containing the character of the irreducible representation $\rho\in R(M_\varphi).$ Since $M_\varphi$ does not contain a closed essential surface (see \cite{Culler-incompressible-1982, Floyd-incompressible-1982}), the dimension of $X$ is at most one.

We now appeal to a result due to Thurston (see \Cref{pro:dimSL_Thurston}) which
implies that the dimension of $X$ is at least one if $\rho(\pi_1(\partial M))$ is not contained in $\{\pm I\}.$ Hence suppose that $\rho(\pi_1(\partial M))$ is contained in $\{\pm I\};$ such a representation is said to have \textbf{trivial peripheral holonomy}.
 Since the longitude is $a^{-1}b^{-1}ab$, we note that $\rho$ is reducible if $\rho(t) = \pm I$ and $\rho(a^{-1}b^{-1}ab) = I.$ Hence $\rho(a^{-1}b^{-1}ab) = -I.$ An elementary calculation setting $\rho(a)$ and $\rho(b)$ as in \Cref{eq:standard_irrep} and solving $\rho(ab) = - \rho(ba)$ shows that the resulting representation satisfies $\tr\rho(a) = \tr\rho(b)=\tr\rho(ab)=0$ and hence is unique up to conjugation. 
  
In particular, $\rho$ descends to an abelian irreducible representation $\overline{\rho} \co \pi_1(\partial M)  \to \PSL(2,\mathbb{C})$ with $\overline{\rho}(t) = \{\pm I\}$ and image isomorphic to the Klein four group, generated by the image of $\pi_1(S).$ It follows that $H_1(M_\varphi; \mathbb{Z}_2) \cong \mathbb{Z}_2^3.$ In particular, \Cref{lem:dimension_bound_order1} implies that the character of $\rho$ is contained in at least one one-dimensional component containing irreducible representations.
\end{proof}

The natural $n$--fold cyclic covering $M_{\varphi^n} \to M_\varphi$ gives an embedding of $\pi_1(M_{\varphi^n})$ as a subgroup of index $n$ in $\pi_1(M_\varphi).$ Hence the restriction map induces a regular map $X(M_\varphi)\to X(M_{\varphi^n}).$ \Cref{exa:A2B-2} below shows that this may not be surjective. We can say more about this map:

\begin{proposition}
The natural $n$--fold cyclic covering $M_{\varphi^n} \to M_\varphi$ induces a regular map $X(M_\varphi)\to X(M_{\varphi^n})$ of degree at most two. If $n$ is odd, then the degree is one. Moreover, the map takes canonical components to canonical components.
\end{proposition}

\begin{proof}
The first sentence follows directly from \cite[Corollary 3.5]{Boyer-algebraic-2002}, and the second from \cite[Proposition 3.3]{Boyer-algebraic-2002}. The last sentence follows since the restriction of a discrete and faithful representation to a subgroup gives a discrete and faithful representation of that subgroup. 
\end{proof}

We conclude with an example that ties the discussion in this section together and exhibits different behaviours. 

\begin{example}
\label{exa:A2B-2}
Let $\varphi = \alpha^2\beta^{-2}.$ Then 
\[
[\varphi_*]= A^2B^{-2} = \begin{pmatrix} 5 & 2 \\ 2 & 1 \end{pmatrix}
\]
and hence $H_1(M_\varphi; \mathbb{Z}_2) \cong \mathbb{Z}_2^3 \cong H_1(M_{\varphi^2}; \mathbb{Z}_2).$ A computation with the set-up in the proof of \Cref{pro:dimension_bound} shows that $X(M_\varphi)$ contains no characters of representations with trivial peripheral holonomy. Indeed, $\overline{X}(M_\varphi)$ contains such a component, but it does not lift. This example fits in the family given in \cite{Heusener-varieties-2004}.

A calculation shows that $X_\varphi (S) = V \cup L_3$, where $V$ is an irreducible curve, and 
$X(M_\varphi) = V' \cup L'_3$ also has two components. The preimage $V'$ of $V$ is the canonical curve and the preimage $L'_3$ of $L_3$ has the property that the trace of the meridian is identically zero on it (hence the meridian maps to an element of order four). All binary dihedral characters are contained in $L'_3$ and since they are simple points they are hence not contained in $V'.$ Since these characters are the only branch points (as will be shown in general), it follows that $L'_3 \to L_3$ is a two-fold branched cover, and $V'\to V$ is a two-fold unbranched cover.

Now $X_{\varphi^2} (S) \supset X_\varphi (S) \cup L_1 \cup L_2 = V \cup L_1 \cup L_2 \cup L_3.$ Hence this is an example where all three lines are contained in the fixed-point set. The preimage in $X(M_{\varphi^2})$ of $V$ is the canonical curve, but the preimage of $L_3$ has two components $L_3^+$ and $L_3^-$ that are characterised by whether the trace of the meridian is identically $+2$ or $-2$. The component $L_3^-$ is the image of $L'_3$ under the map
$X(M_\varphi)\to X(M_{\varphi^2})$ since the square of an element of order four has order two.
Hence $L_3^+$ is not contained in the image of the map.
We note that each map $L_3^-\to L_3$ and  $L_3^+\to L_3$ is one-to-one.
\end{example}

 
\subsection{Reducible characters}
\label{subsec:Reducible characters}

Since $H_1(M_\varphi,\mathbb{Z}) \cong \mathbb{Z} \oplus \text{Tor}(M_\varphi,\mathbb{Z})$, it follows  that $X^{\text{red}}(M_\varphi)$ consists of finitely many affine lines.

From the set-up in the proof of \Cref{lem:im(r)}, we see that the reducible characters in $X_\varphi(S)$ that are restrictions of reducible characters in $X(M_\varphi)$ are the same as characters of abelian representations $\rho$ satisfying $\rho(\varphi_*(\gamma)) = \rho(\gamma)$ for all $\gamma\in \pi_1(S).$ 
Since the monodromy of $M_\varphi$ is $\varphi_* = \begin{pmatrix}k_1&k_2\\k_3&k_4\end{pmatrix}\in SL(2,\mathbb{Z})$ with $|k_1+k_4|>2$, all possible pairs $(s,p)$ of eigenvalues of $\rho(a)$ and $\rho(b)$ corresponding to a common invariant subspace are determined by 
\begin{equation}\label{eq:eigenvalues_red}
    \begin{cases}
    s = s^{k_1}p^{k_3} \\
    p = s^{k_2}p^{k_4}
    \end{cases}
\end{equation}
The (not necessarily distinct) components of $X^{\text{red}}(M_\varphi)$ corresponding to the finitely many solutions to these equations are given by 
$$X_{s,p} = \{\;(s+s^{-1},p+p^{-1},sp+s^{-1}p^{-1},m+m^{-1},sm+s^{-1}m^{-1},pm+p^{-1}m^{-1},spm+s^{-1}p^{-1}m^{-1})\;|\;m\in\mathbb{C}^*\;\}$$ 
In particular, each component of $X^{\text{red}}(M_\varphi)$ has image a single point in $X_\varphi(S).$

The proof of the following proposition is an application of the main result of \cite{Heusener-deformation-2005} and generalises the example given in that paper by using fundamental results due to Fox~\cite{Fox-free_calculus-1953}. We thank Michael Heusener and Joan Porti for encouraging correspondence, and refer the reader to \cite{Heusener-deformation-2005} for the required definitions of twisted Alexander polynomials.

\begin{proposition}\label{pro:intersection_red_irrep}
For every component $X_{s_0,p_0}$ of $X^{\text{red}}(M_\varphi)$, there exists $m_0\in\mathbb{C}^*$ such that $X_{s,p}$ intersects $X^{\text{irr}}(M_\varphi)$ at $\chi_{m_0}$ where $\chi_{m_0}$ is the character of the abelian representation $\rho_{m_0}$ with \[
\rho_{m_0}(t) =\Diag(m_0,m_0^{-1}),\; \rho_{m_0}(a) =\Diag(s_0,s_0^{-1}),\; \rho_{m_0}(b) =\Diag(p_0,p_0^{-1})
\]
Moreover, $\chi_{m_0}$ is contained on exactly one curve $C$ in $X^{\text{irr}}(M_\varphi),$ a smooth point of both $X_{s_0,p_0}$ and $C$ and the intersection at $\chi_{m_0}$ is transverse. Moreover, if $s_0, t_0 \in \{\pm 1\},$ then $m_0$ is the square root of a zero of the characteristic polynomial of $\varphi_*$ and otherwise $m_0=\pm 1.$
\end{proposition}

\begin{proof}
Let $m = m_0^2,$ $s = s_0^2$ and $p=p_0^2.$ With this set up, the remainder of the proof applies verbatim to the $\PSL(2,\mathbb{C})$--character variety.
We have a group homomorphism $\alpha\co\pi_1(M_\varphi) \to \mathbb{C}^*$ with 
\[
\alpha(a) = s,\; \alpha(b) = p,\; \alpha(t) = m
\]
 The pair $(s,p)$ satisfies \Cref{eq:eigenvalues_red}, that is, the restriction $\alpha\mid_{\text{Tor}(M_\varphi,\mathbb{Z})}$ is induced by a group homomorphism $h:\langle a,b\rangle\to U(1)$ such that $h\circ \varphi = h$ and $h(a)=s$ and $h(b)=p.$

We apply the criterion of deformations of reducible characters in \cite[Theorem 1.3]{Heusener-deformation-2005}. Fix the canonical splitting $H_1(M_\varphi,\mathbb{Z})\cong \mathbb{Z}\oplus\text{Tor}(M_\varphi,\mathbb{Z}) \to \text{Tor}(M_\varphi,\mathbb{Z})$ where $p(a) = a$, $p(b)=b$ and $p(t) = 0$ and an element $\psi\in H^{1}(M_\varphi,\mathbb{Z})$ where $\psi(a)=\psi(b) = 0$ and $\psi(t) = 1$. To prove the statement, we only need to show that $m$ is a simple zero of the twisted Alexander polynomial $\Delta^{\psi_\alpha}_{M_\varphi}$. A direct computation shows that the Jacobian matrix is 
\[
J_{M_\varphi} = \begin{pmatrix}
t^{-1}(a^{-1}-1) & t^{-1}a^{-1}\bigg(t\dfrac{\partial}{\partial a}\varphi(a)-1\bigg) & t^{-1}a^{-1}t\dfrac{\partial}{\partial b}\varphi(a)\\
t^{-1}(b^{-1}-1) & t^{-1}b^{-1}t\dfrac{\partial}{\partial a}\varphi(b)& t^{-1}b^{-1}\bigg( t\dfrac{\partial}{\partial b}\varphi(b)-1\bigg)
\end{pmatrix}
\]
Each $\varphi\in\Aut(\langle a,b\rangle)$ linearly extends to a ring homomorphism of $\mathbb{Z}\langle a,b\rangle$, also denoted by $\varphi$. Under the abelianisation homomorphism, it also induces a homomorphism $\varphi^\text{ab}:\mathbb{Z}[a,b]\to \mathbb{Z}[a,b]$. The Jacobian of $\varphi$ is
\[
J_\varphi = \begin{pmatrix}
\frac{\partial}{\partial a}  \varphi(a) & \frac{\partial}{\partial b} \varphi(a)\\
\frac{\partial }{\partial a}\varphi(b) & \frac{\partial }{\partial b}\varphi(b)
\end{pmatrix}
\]
Let $J^\text{ab}_\varphi$ denotes its image under the abelianisation homomorphism. By the chain rule~\cite{Fox-free_calculus-1953}, $J_{\varphi_1\varphi_2} = \varphi_1(J_{\varphi_2}) J_{\varphi_1}$ for any $\varphi_1,\varphi_2\in \Aut(F_2)$. Applying the abelianisation homomorphism and taking determinants on both sides, we have
\begin{align*}
    \det J_{\varphi_1\varphi_2}^\text{ab} &= \det \varphi_1^\text{ab}(J_{\varphi_2}^\text{ab}) \det J_{\varphi_1}^\text{ab}\\
    &=  \varphi_1^\text{ab}(\det J_{\varphi_2}^\text{ab}) \det J_{\varphi_1}^\text{ab}
\end{align*}
We may assume that $\varphi_*$ is given as in \Cref{eq:positive_standard_word}. Then $\det J_\alpha^\text{ab} = \det J_{\beta^{-1}}^\text{ab} =1$ implies that $\det J_{\varphi}^\text{ab} = 1$. Let $J^h_\varphi$ denotes the image of $J_\varphi$ under the map $h$. It is clear that $h(J_{\varphi}^\text{ab}) = J^h_\varphi$. Hence $\det J^h_\varphi = 1$.\\
By fundamental formula of free calculus~\cite{Fox-free_calculus-1953}, we have 
\[
\varphi(a) - 1 = \bigg(\dfrac{\partial}{\partial a}\varphi(a)\bigg)(a-1) + \bigg(\dfrac{\partial}{\partial b}\varphi(a)\bigg)(b-1)
\] Applying $h$ on both sides, we have
\[ (h\bigg(\dfrac{\partial}{\partial a}\varphi(a)\bigg)-1)(s-1) + h\bigg(\dfrac{\partial}{\partial b}\varphi(a)\bigg)(p-1) = 0\]
since $h\circ \varphi = h$. By symmetry,
\[ (h\bigg(\dfrac{\partial}{\partial b}\varphi(b)\bigg)-1)(p-1) + h\bigg(\dfrac{\partial}{\partial a}\varphi(b)\bigg)(s-1) = 0\]

Define $\psi_{\alpha}:\pi_1(M_\varphi)\to \mathbb{C}[x^{\pm 1}]$ where $\psi_{\alpha}(\gamma) = \alpha(\gamma)x^{\psi(\gamma)}$. The Jacobian $J^{\psi_\alpha}_{M_\varphi}$ is given by $(J^{\psi_\alpha}_{M_\varphi})_{ij} = \psi_{\alpha}((J_{M_\varphi})_{ij})$. Since Alexander invariants are obtained from determinants, we may multiply each row by an unit and work with the following matrix:
\[
 \begin{pmatrix}
1-s & h\bigg(\dfrac{\partial \varphi(a)}{\partial a}\bigg)x-1 & h\bigg(\dfrac{\partial }{\partial b}\varphi(a)\bigg)x\\
1-p & h\bigg(\dfrac{\partial \varphi(b)}{\partial a}\bigg)x & h\bigg(\dfrac{\partial }{\partial b}\varphi(b)\bigg)x-1
\end{pmatrix}
\]
We now have the following cases:
\begin{enumerate}
    \item If $h(a)=h(b) = 1$, a direct computation shows that $J^h_\varphi = \varphi_*.$ So $\Delta^{\psi_\alpha}_{M_\varphi}$ is the characteristic polynomial of $\varphi_*$, which has $2$ distinct simple zeroes since the trace is distinct from $\pm 2.$ Hence $m$ equals one of these roots.
    \item If $h(a) = 1$ and $h(b)\neq 1$, from the above equations, we have $h\bigg(\dfrac{\partial}{\partial b} \varphi(a)\bigg) = 0$ and $h\bigg(\dfrac{\partial }{\partial b}\varphi(b)\bigg) =1$. 
    Combined with $\det J^h_\varphi = 1$, we have $h\bigg(\dfrac{\partial }{\partial a}\varphi(a)\bigg) =1$ so $\Delta^{\psi_\alpha}_{M_\varphi} = x-1$. By symmetry, $\Delta^{\psi_\alpha}_{M_\varphi} = x-1$ if $h(a)\neq 1$ and $h(b)=1$. Hence in each of these  cases, $m=1.$
    \item If $h(a)\neq 1$ and $h(b)\neq 1$, we have 
    \[
    (h\bigg(\dfrac{\partial }{\partial a}\varphi(a)\bigg)-1)(h\bigg(\dfrac{\partial }{\partial b}\varphi(b)\bigg)-1) 
    = h\bigg(\dfrac{\partial }{\partial b}\varphi(a)\bigg)h\bigg(\dfrac{\partial }{\partial a}\varphi(b)\bigg)
    \]
     by multiplying the two equations. Rearranging the above equation and substitute $\det J^h_\varphi =1$ into it, we have
     $h\bigg(\dfrac{\partial }{\partial a}\varphi(a)\bigg) + h\bigg(\dfrac{\partial }{\partial b}\varphi(b)\bigg) = 2$. Then $h\bigg(\dfrac{\partial }{\partial b}\varphi(b)\bigg) - 1 =  1 - h\bigg(\dfrac{\partial }{\partial a}\varphi(a)\bigg)$ and we can rewrite the second equation as
     \[ (h\bigg(\dfrac{\partial }{\partial a}\varphi(a)\bigg)-1)(p-1) = h\bigg(\dfrac{\partial}{\partial a}\varphi(b)\bigg)(s-1)\]
     This gives us that the determinant of the first two columns of $J_{M_\varphi}^{\psi_\alpha}$ is $(1-p)(x-1)$. We can perform similar computation for the other two determinants which both have a factor $x-1$. This implies $\Delta^{\psi_\alpha}_{M_\varphi} = x-1$ and $m=1.$
\end{enumerate}
This completes the proof.
\end{proof}
A direct consequence of the above proposition is that $r(X(M_\varphi)) = r(X^{\text{irr}}(M_\varphi)).$ We now show that there are no ramification points of $r\co X^{\text{irr}}(M_\varphi) \to X_\varphi(S)$ at reducible characters contained on $X^{\text{irr}}(M_\varphi)$.

Consider the action of $H^1(M_\varphi,\mathbb{Z}_2 )$ on $X^{\text{red}}(M_\varphi)$. We first characterise the points in $X^{\text{red}}(M_\varphi)$ with non-trivial stabilisers under this action.
Suppose $h\in H^1(M_\varphi,\mathbb{Z}_2 )$ with $h(t)=\epsilon_t, h(a)=\epsilon_a, h(b)=\epsilon_b.$ Then a direct calculation shows that $S(h)\cap X^{\text{red}}(M_\varphi)$ consists of all points in $X^{\text{red}}(M_\varphi)$ satisfying 
\[m^2=\epsilon_t, \; s^2=\epsilon_a, \; p^2=\epsilon_b\]
In particular, the trace functions take values in $\{-2, 0, 2\}.$ 

We now focus on $S(\varepsilon)$. Recall that $\varepsilon(t)=-1, \varepsilon(a)=\varepsilon(b)=1.$
In this case, 
\[m^2=-1, \; s^2=1, \; p^2=1\]
and hence $s, p \in \{\pm 1\}.$
If such a character is in the intersection of 
$X^{\text{red}}(M_\varphi) \cap X^{\text{irr}}(M_\varphi),$ then \Cref{lem:porti} implies that there is such a character of a non-abelian reducible representation $\rho.$
Up to conjugacy and the action of $\varepsilon$ on $R(M_\varphi),$ we may assume that 
$$
\begin{matrix}
  \rho(t) = \begin{pmatrix}i&0\\0&-i\end{pmatrix} 
& 
  \rho(a) = \begin{pmatrix}s& u \\ 0 &s\end{pmatrix}   
&  
  \rho(b) = \begin{pmatrix}p &v\\0&p\end{pmatrix}
\end{matrix}
$$ where $s, p \in \{\pm 1\}$ and $(u,v)\neq (0,0)$.
Then
$$
\begin{pmatrix}s& -u \\ 0 &s\end{pmatrix}  
= \rho(t^{-1}at)
= \rho(\varphi(a))
= \begin{pmatrix}
s^{k_1}\epsilon_b^{k_3}& k_3s^{k_1}p^{k_3-1}v+  k_1s^{k_1-1}p^{k_3}u\\ 
0 &s^{k_1}p^{k_3}
\end{pmatrix}  
$$
and
$$
\begin{pmatrix}p& -v \\ 0 &p\end{pmatrix}  
= \rho(t^{-1}bt)
= \rho(\varphi(b))
= \begin{pmatrix}
s^{k_2}p^{k_4}& k_4s^{k_2}p^{k_4-1}v+  k_2s^{k_2-1}p^{k_4}u\\ 
0 &s^{k_2}p^{k_4}
\end{pmatrix}  
$$
This implies in particular the equations $u = -sp(k_3v+k_1u)$ and $v = -sp(k_4v+k_2u).$
Now $sp= \pm 1$, and hence these two equations give
\[
\text{transpose}[\varphi_*] \begin{pmatrix}u \\ v\end{pmatrix}   
=
\begin{pmatrix}k_1& k_3 \\ k_2 &k_4\end{pmatrix}   \begin{pmatrix}u \\ v\end{pmatrix}   
= \pm \begin{pmatrix}u \\ v\end{pmatrix}   \neq \begin{pmatrix}0 \\ 0\end{pmatrix}  
\]
But then $\varphi_*$ has eigenvalue $\pm 1$ and hence
$\tr(\varphi_*) = \pm 2,$ which is a contradiction. This shows that there are no ramification points of $r\co X^{\text{irr}}(M_\varphi) \to X_\varphi(S)$ at reducible characters contained on $X^{\text{irr}}(M_\varphi)$.


\subsection{Irreducible characters}

The following lemma proves that every irreducible character in $X_\varphi(S)$ is the restriction of two distinct irreducible characters in $X(M_\varphi)$, and hence is not a branch point of $r\co X(M_\varphi) \to X_\varphi(S).$ 
Moreover, \Cref{exa:A2B-2} gives examples where these distinct characters lie in different Zariski components of $X(M_\varphi)$ and where they lie in the same Zariski component.

\begin{lemma}
\label{lem:irreps in fixed point set}
For any irreducible character $\chi_\rho$ of $S$ such that $\chi_\rho\in X_\varphi(S)$, there exists  $T\in \SL(2,\mathbb{C})$ such that $T^{-1}\rho(\gamma)T = \rho(\varphi_*(\gamma))$ for all $\gamma\in \pi_1\, S$. Moreover, $T$ is unique up to sign and at least one of $\tr\,T$, $\tr\,(\rho(a)T)$, $\tr\,(\rho(b)T)$ and $\tr\,(\rho(ab)T)$ is nonzero.
\end{lemma}

\begin{proof}
Let $\rho\in{R}(S)$ be an irreducible representation with $\chi_\rho\in X_{\varphi}S$. We define $\rho_1\in{R}(S)$ by $\rho_1(\gamma) = \rho(\varphi_*(\gamma))$ for $\gamma\in \pi_1 S$. According to the definition of $X_\varphi(S)$, $\chi_\rho = \chi_{\rho_1}$. Since $\chi_\rho$ is the character of an irreducible representation, it follows that $\rho$ and $\rho_1$ are conjugate. 
Hence there exists $T\in SL(2,\mathbb{C})$ such that $T^{-1}\rho T = \rho_1$. Since $\chi$ is irreducible, the centraliser of $\rho$ is $\pm E,$ where $E$ is the identity matrix. Hence $T$ is unique up to sign. This proves the first part of the lemma, and it remains to show that at least one of the stated traces is non-zero.

Up to conjugation, we may assume that
$$\begin{matrix}
  \rho(a) = \begin{pmatrix}s&0\\1&s^{-1}\end{pmatrix}  & \rho(b) = \begin{pmatrix}p&u\\0&p^{-1}\end{pmatrix}
\end{matrix}$$ where $u\neq 0$.
Since $\chi_\rho\in X_{\varphi}S$, we have 
$$\begin{matrix}
  \rho(\varphi_*(a)) = \begin{pmatrix}s_1&s_2\\s_3&s+s^{-1}-s_1\end{pmatrix}  & \rho(\varphi_*(b)) = \begin{pmatrix}p_1&p_2\\p_3&p+p^{-1}-p_1\end{pmatrix}
\end{matrix}$$
for some $s_1, s_2, s_3, p_1, p_2, p_3 \in \mathbb{C}.$
Suppose $T=\begin{pmatrix}m_1&m_2\\m_3&m_4\end{pmatrix}$ for some $m_1, m_1, m_3, m_4 \in \mathbb{C},$ and that $\tr\,T = \tr\,(\rho(a)T)=\tr\,(\rho(b)T)=\tr\,(\rho(ab)T) = 0$ holds. Then, we have 
$$\begin{cases}
\;0=m_1+m_4 \\
\;0=sm_1 + m_2 + s^{-1}m_4 \\
\;0=pm_1 + um_3 + p^{-1}m_4 \\
\;0=spm_1 + pm_2 + usm_3 + (u+s^{-1}p^{-1})m_4 
\end{cases}$$
Note that the equations do not change if one replaces $T$ with $-T.$
From the first three equations, we have $m_4 = -m_1$, $m_2 = (s^{-1}-s)m_1$ and $m_3 = \frac{(p^{-1}-p)m_1}{u}$, which gives
$$T=m_1\begin{pmatrix}1&s^{-1}-s\\\frac{p^{-1}-p}{u}&-1\end{pmatrix}$$
Substituting $m_2,m_3,m_4$ into the fourth equation, we have 
$$(u + sp + s^{-1}p^{-1} - s^{-1}p - sp^{-1})m_1 = 0$$
However, this implies the contradiction
$$\det T = -\dfrac{(u + sp + s^{-1}p^{-1} - s^{-1}p - sp^{-1})m_1}{u} = 0$$
Hence at least one of $\tr\,T$,$\tr\,(\rho(a)T)$,$\tr\,(\rho(b)T)$ and $\tr\,(\rho(ab)T)$ is nonzero.
\end{proof}

\begin{proposition}
\label{cor:irreps in fixed point set}
The fibres of $r\co X^{\text{irr}}(M_\varphi)\to X_\varphi(S)$ are the orbits of $\varepsilon$ and the branch points for $r$ are contained in the set of reducible characters in $X_\varphi(S).$ Moreover, the ramification points are simple points of 
$X^{\text{irr}}(M_\varphi)$ and are precisely the binary dihedral characters in $X(M_\varphi(\lambda))\subset X(M_\varphi).$
\end{proposition}

\begin{proof}
The action of $\varepsilon$ with respect to the affine coordinate $(x,y,z,u,v,w,q)$ in $\mathbb{C}^{7}$ is given by the the involution $(x,y,z,u,v,w,q)\to (x,y,z,-u,-v,-w,-q).$ It was shown in \Cref{lem:irreps in fixed point set} that the preimage under $r$ of each irreducible character in $X_\varphi(S)$ is the orbit under this involution and has precisely two elements. It follows by continuity that each fixed point of $\varepsilon$ is a ramification point for $r.$
Hence the branch points for $r$ are contained in the set of reducible characters in $X_\varphi(S).$

It was shown in \Cref{subsec:Reducible characters} that there are no ramification points of $r$ in $X^{\text{red}}(M_\varphi) \cap X^{\text{irr}}(M_\varphi).$ Hence all ramification points correspond to irreducible characters in $X^{\text{irr}}(M_\varphi)$  that restrict to reducible characters in $X_\varphi(S)$ and are fixed by $\varepsilon.$ It was noted in \Cref{schol:binary_dihedral} that the irreducible characters $X^{\text{irr}}(M_\varphi)$  that restrict to reducible characters in $X_\varphi(S)$ are precisely the binary dihedral characters in $X(M_\varphi(\lambda))\subset X(M_\varphi)$, and it is an elementary calculation to verify that they satisfy $u=v=w=q=0.$

It is shown in \cite[Proposition 5.3]{Boyer-algebraic-2002} that these binary dihedral characters are simple points or $X(M_\varphi).$
\end{proof}

\begin{proof}[Proof of \Cref{thm:main}]
First note that \Cref{pro:intersection_red_irrep} implies that $r(X(M_\varphi)) = r(X^{\text{irr}}(M_\varphi)).$
The main statement of \Cref{thm:main} is now the content of \Cref{pro:dimension_bound,cor:irreps in fixed point set}. The additional information about the number of branch points follows from \cite{Boyer-algebraic-2002}, but the formula is different. We now justify our formulation.
The number of binary dihedral characters in $X(M_\varphi)$ is given by \cite[Propositions 5.3 and 4.5]{Boyer-algebraic-2002} as
\[
\frac{1}{2}\big( |2 + \tr(\varphi_*)| - \zeta_{\varphi_*} \big)
\]
where $\zeta_{\varphi_*}\in \{1, 2, 4\}$ is the order of  $\Hom(\,\coker(\varphi_*+1_{H_1(S^1\times S^1)},\, \mathbb{Z}_2).$ Here, the once-punctured torus is identified with the complement of a point in $S^1\times S^1.$ A direct calculation shows that 
\begin{align*}
\zeta_{\varphi_*}= 4 \iff o(\varphi_2)=1 \iff b_1(\varphi) = 3\\
\zeta_{\varphi_*}= 2 \iff o(\varphi_2)=2 \iff b_1(\varphi) = 2\\
\zeta_{\varphi_*}= 1 \iff o(\varphi_2)=3 \iff b_1(\varphi) = 1\\
\end{align*}
Hence the identity $\zeta_{\varphi_*} = 2^{3 - o(\varphi_2)} = 2^{b_1(\varphi) -1}$ gives
\[
\frac{1}{2}\big( |2 + \tr(\varphi_*)| - \zeta_{\varphi_*} \big) = \frac{1}{2} |2 + \tr(\varphi_*)| - 2^{{b_1(\varphi)-2}}
\]
as claimed in \Cref{eq:binary_count}.
\end{proof}

The relationship between Euler characteristic and genus of curves gives the following.

\begin{corollary}
\label{cor:genus_bounds}
Suppose $X_\varphi(S)$ and $X^{\text{irr}}(M_\varphi)$ are irreducible, and that $X_\varphi(S)$ is a non-singular affine curve. Let $g_0$ and $g_1$ denote the genera of the smooth projective models of $X^{\text{irr}}(M_\varphi)$ and $X_\varphi(S)$ respectively. Then $$g_0 = 2g_1 - 1 + \frac{e_1+e_\infty}{2}$$ where $e_1= \frac{1}{2}|2 + \tr(\varphi_*)| - 2^{{b_1(\varphi)}-2}$ is the number of branch points in $X_\varphi(S)$ and $e_\infty$ is number of branch points at ideal points of $X_\varphi(S)$. In particular, if $i_\infty$ is the total number of ideal points of $X_\varphi(S)$, then $g_0$ is bounded by
    $$2g_1 - 1 + \frac{e_1}{2} \leq g_0 \leq 2g_1 - 1 + \frac{e_1}{2} +\frac{i_\infty}{2}$$
\end{corollary}


\section{Projective characters}
\label{sec:Projective characters}

Following the blueprint of \Cref{sec: char of once-punctured torus bundles}, we now analyse the $\PSL(2, \mathbb{C})$--character varieties of once-punctured torus bundles. We make frequent use of Heusener and Porti~\cite{Heusener-varieties-2004} and refer the reader to original sources therein.


\subsection{The restriction map}

Let $\overline{X}_\varphi(S) = \{ \chi_{\prho} \in \overline{X}(S) \mid \chi_{\prho} = \chi_{\prho\varphi} \}$. 
We first show that the natural restriction map $\overline{r} \co \overline{X}(M_\varphi) \to \overline{X}(S)$ again satisfies $\im (\overline{r})=\overline{X}_\varphi(S).$ 
Recall that we have the four-fold branched covering map $X(S)\to \overline{X}(S)$ from \Cref{sec:charF2},
\[
\overline{X}(S) \ni (x, y, z) \mapsto (x^2, y^2, z^2, xyz) \in \overline{X}(S)
\]
Let $\overline{\rho}\co \pi_1(M_\varphi)\to \PSL(2,\mathbb{C})$ be a representation. Let $T, A, B\in \SL(2,\mathbb{C})$ satisfying $\overline{\rho}(t) = \pm T,$ $\overline{\rho}(a) = \pm A,$ $\overline{\rho}(b) = \pm B$. To account for all possible choices of sign, we write $\varphi(a) = \varphi_a(a,b)$ for a fixed word given for $\varphi(a)$ in terms of the generators. Similarly for $b.$
Then
\begin{align}\label{eq:PSL_setup_A}
T^{-1}AT&= \varepsilon_a(A,B)\; \varphi_a(A, B) \\
\label{eq:PSL_setup_B}
T^{-1}BT&= \varepsilon_b(A,B)\; \varphi_b(A,B)
\end{align}
where $\varepsilon_a(A,B), \varepsilon_b(A,B)\in \{\pm 1\}$. This implies $\im (\overline{r}) \subseteq \overline{X}_\varphi(S).$ 

Conversely, let $\prho \co \pi_1(S) \to \pi_1(S)$ such that $\chi_{\prho} \in \overline{X}_\varphi(S)$. Suppose $\prho(a) = \pm A$ and $\prho(b) =\pm B,$ where $A, B\in \SL(2,\mathbb{C}).$
Since $\chi_{\prho} = \chi_{\prho\varphi}$, the first three coordinates of the corresponding points in $ \overline{X}(S)$ imply that
there are $\varepsilon_a, \varepsilon_b, \varepsilon_{ab}\in \{\pm 1\}$ such that
\[
\tr A = \varepsilon_a\tr (\varphi_a(A,B)), \qquad \tr B = \varepsilon_b \tr (\varphi_b(A,B))), \qquad
\tr AB = \varepsilon_{ab} \tr(\; \varphi_a(A,B))\varphi_b(A,B)))
\]
Now the fourth coordinate implies that $\varepsilon_a\varepsilon_b\varepsilon_{ab}=1$ and hence $\varepsilon_{ab} = \varepsilon_a\varepsilon_b$. Hence
\[
\tr A = \tr (\varepsilon_a\varphi_a(A,B)), \qquad \tr B =  \tr (\varepsilon_b\varphi_b(A,B))), \qquad
\tr AB =  \tr(\; \varepsilon_a\varphi_a(A,B))\cdot \varepsilon_b\varphi_b(A,B)))
\]
As in the proof of \Cref{lem:im(r)} there is $T \in \SL(2,\mathbb{C})$ with 
\[
T^{-1}AT= \varepsilon_a\;\varphi_a(A, B) \qquad\text{and}\qquad
T^{-1}BT= \varepsilon_b\;\varphi_b(A,B)
\]
This shows that $\chi_{\prho} \in \im (\overline{r}).$ Hence $\im (\overline{r})=\overline{X}_\varphi(S)$ as claimed.

It follows as in the proof of \Cref{lem:irreps in fixed point set} that for an irreducible representation 
$\prho \co \pi_1(M_\varphi) \to \PSL(2,\mathbb{C})$ 
\begin{equation}
\label{eq:PSL_characters}
(\; (\tr\overline{\rho}(t))^2, \quad
(\tr\overline{\rho}(ta))^2, \quad 
 (\tr\overline{\rho}(tb))^2, \quad
 (\tr\overline{\rho}(tab))^2 \;)\neq (0,0,0,0)
\end{equation}

As in \Cref{subsec:Reducible characters}, it is easy to show that there are finitely many characters in $\overline{X}_\varphi(S)$ that arise from reducible representations of $\pi_1(M_\varphi).$ The proof of \Cref{pro:intersection_red_irrep} shows that $\overline{r} \co \overline{X}^{irr}(M_\varphi) \to \overline{X}_\varphi(S)$ is surjective.

The proof of \Cref{cor:iso for b1=1} and the statement of \Cref{pro:dimension_bound} now give the following:

\begin{proposition}
\label{pro:diagram_rank1}
Suppose $M_\varphi$ is a hyperbolic once-punctured torus bundle with $b_1(\varphi)=1.$ 
Then we have the following commutative diagram of surjective maps whose degrees are indicated in the diagram:
\[
\xymatrix{
X^{\text{irr}}(M_\varphi) \ar[d]^{2:1}_{q_1} \ar[r]^{r}_{2:1} 	& X_\varphi(S) \ar[d]^{1:1}_{q_2} \ar@{}[r]|-*[@]{\subset} 	& X(S) \ar[d]^{4:1}_{q_2} \ar@{}[r]|-*[@]{\cong} & \mathbb{C}^3\\
\overline{X}^{\text{irr}}(M_\varphi) \ar[r]^{\overline{r}}_{1:1} 		& \overline{X}_\varphi(S)  \ar@{}[r]|-*[@]{\subset}	& \overline{X}(S)	
}		
\]
In particular, every component of $\overline{X}^{\text{irr}}(M_\varphi)$ is one-dimensional.
\end{proposition}

For the bundles with $b_1(\varphi)>1,$ we have the following to consider.
Let $\overline{\rho}\co \pi_1(S) \to \PSL(2, \mathbb{C})$ be a representation with $\chi_{\prho} \in \overline{X}_\varphi(S).$ We call the characters in $\overline{r}^{-1}(\chi_{\overline{\rho}})$ \textbf{extensions} of $\chi_{\overline{\rho}},$ and similar for representations. Any two extensions of $\prho$ to $\pi_1(M_\varphi)$ differ by an element in the centraliser $C(\im(\overline{\rho})).$
It is well-known (see \cite{Heusener-varieties-2004}) that if $\prho$ is irreducible then
the centraliser $C(\im(\overline{\rho}))$ is either trivial, or cyclic of order two, or isomorphic to the Klein four group. This is unlike the situation for $\SL(2,\mathbb{C}),$ where the centraliser of an irreducible representation is always central. 
In terms of our coordinate system for $\overline{X}(S)$ these possibilities are given in \Cref{sec:PSLorigin_axes}.
We start with the following preliminary observation.

\begin{lemma}
\label{lem:H-orbits or not}
Suppose $\chi_\rho, \chi_\sigma \in X_\varphi(S)$ are in the same $H^1(S, \mathbb{Z}_2)$--orbit, and that $\chi_\rho$ is irreducible. 
If $\chi_\rho$ and $\chi_\sigma$ are in the same $H$--orbit, then $q_1r^{-1}(\chi_\rho)=q_1r^{-1}(\chi_\sigma).$ If they are in distinct $H$--orbits, then $q_1r^{-1}(\chi_\rho)\neq q_1r^{-1}(\chi_\sigma).$ Moreover, $q_1r^{-1}(\chi_\rho)$ and $q_1r^{-1}(\chi_\sigma)$ are the characters of representations into $ \PSL(2,\mathbb{C})$ that agree on $\pi_1(S)$ and whose image of $t$ differs by a non-trivial element in the centraliser of the image of $\pi_1(S).$
\end{lemma}

\begin{proof}
Let $\rho\co\pi_1(S)\to \SL(2,\mathbb{C})$ be a representation with character $\chi_\rho.$ Then there is $h\in H^1(S, \mathbb{Z}_2)$ such that the character of $\sigma(\gamma) =  h(\gamma) \rho(\gamma)$ is $\chi_\sigma.$

Since $\chi_\rho \in X_\varphi(S),$ there is $T \in \SL(2,\mathbb{C})$ such that $T^{-1}\rho(\gamma)T = \rho(\varphi(\gamma))$ for all $\gamma \in \pi_1(S).$ In particular, $\rho$ extends to a representation of $\pi_1(M_\varphi)$ into $\SL(2,\mathbb{C})$ by letting $\rho(t) = T$ or $\rho(t) = -T$ and these are the only extensions according to \Cref{lem:irreps in fixed point set}. Note that these extensions are in the same $\langle \varepsilon \rangle$--orbit.

If $h \in H,$ then $h(\gamma) = h(\varphi(\gamma))$ for all $\gamma \in \pi_1(S).$ Hence 
\[
T^{-1}\sigma(\gamma)T = T^{-1}h(\gamma)\rho(\gamma)T =h(\gamma)T^{-1}\rho(\gamma)T =  h(\gamma)\rho(\varphi(\gamma))
=  h(\varphi(\gamma))\rho(\varphi(\gamma)) = \sigma(\varphi(\gamma))
\]
for all $\gamma \in \pi_1(S).$ So as above, $\sigma$ extends to a representation of $\pi_1(M_\varphi)$ by letting $\sigma(t) = T$ or $\sigma(t) = -T$. Hence the extensions of $\rho$ are in the same  $H^1(M_\varphi, \mathbb{Z}_2)$--orbit as the extensions of $\sigma.$ This implies 
$q_1r^{-1}(\chi_\rho)=q_1r^{-1}(\chi_\sigma).$

Suppose there is no $h\in H^1(S, \mathbb{Z}_2)$ such that the character of $\sigma(\gamma) =  h(\gamma) \rho(\gamma)$ is $\chi_\sigma.$ Then $q_1r^{-1}(\chi_\rho)\neq q_1r^{-1}(\chi_\sigma).$ Since 
$q_2(\chi_\rho)=q_2(\chi_\sigma),$ we have 
$\overline{r}q_1r^{-1}(\chi_\rho)= q_2(\chi_\rho)=q_2(\chi_\sigma)=\overline{r}q_1r^{-1}(\chi_\sigma).$
So $q_1r^{-1}(\chi_\rho)$ and $q_1r^{-1}(\chi_\sigma)$ are the characters of non-conjugate representations that agree on $\pi_1(S).$ Hence the conclusion.
\end{proof}


\subsection{The origin and three coordinate axes (revisited)}
\label{sec:PSLorigin_axes}

Let $\overline{\rho}\co \pi_1(S) \to \PSL(2, \mathbb{C})$ be an irreducible representation.
We have $C(\im(\overline{\rho}))\cong \mathbb{Z}_2\oplus\mathbb{Z}_2$ if and only if $\chi_{\prho} = (0,0,0,0)\in \overline{X}(S).$ For every $\varphi$, we have $(0,0,0) \in X_\varphi(S)$ and hence $(0,0,0,0)\in \overline{X}_\varphi(S).$
Up to conjugation, we may assume
\begin{equation}\label{eq:paramK4}
    \overline{\rho}(a) = \pm\begin{pmatrix}
    0&1\\-1&0
    \end{pmatrix} = \kappa_1
    \qquad
    \overline{\rho}(b) = \pm\begin{pmatrix}
    0&i\\i&0
    \end{pmatrix}= \kappa_2
     \qquad\
    \overline{\rho}(ab) = \pm\begin{pmatrix}
    i&0\\0&-i
    \end{pmatrix}= \kappa_3
\end{equation}
These M\"obius transformations represent rotations by $\pi$ with respective axes $[-i, i]$, $[-1,1]$ and $[0,\infty],$ and we have  $C(\im(\overline{\rho})) = \im(\overline{\rho}).$

\begin{lemma}
\label{lem:K4_number_of_lifts}
If $\overline{\rho}\co \pi_1(S) \to \PSL(2, \mathbb{C})$ is irreducible and
$|C(\im(\overline{\rho}))|=4,$ then 
$|\overline{r}^{-1}(\chi_{\overline{\rho}})| = 2^{b_1(\varphi)-1}$. 
Moreover, exactly of the characters in $\overline{r}^{-1}(\chi_{\overline{\rho}})$ lifts to $X(M_\varphi).$
For each representation of $\pi_1(M_\varphi)$ with character in $\overline{r}^{-1}(\chi_{\overline{\rho}}),$ the image of the longitude is always trivial; the image of the meridian has order two or four if $b_1(\varphi)=2$ and it has order two or is trivial if $b_1(\varphi)=3.$
\end{lemma}

\begin{proof}
Since $(0,0,0)\in {X}_\varphi(S)$ is the unique preimage of $\chi_{\prho}$ in ${X}_\varphi(S)$, it follows that any two lifts  to $\SL(2, \mathbb{C})$ of the Klein four group are conjugate in $\SL(2, \mathbb{C})$. Hence the uniqueness up to sign in \Cref{lem:irreps in fixed point set} implies that exactly one of the characters in $\overline{r}^{-1}(\chi_{\overline{\rho}})$ lifts to $X(M_\varphi).$ 

Note that the image of the longitude is the commutator of $\kappa_1$ and $\kappa_2$ and hence trivial.

If $b_1(\varphi)=3,$ then we have $\varphi(\overline{\rho}(a)) = \overline{\rho}(a)$ and 
$\varphi(\overline{\rho}(b)) = \overline{\rho}(b)$ and hence 
$\overline{\rho}(t) \in C(\im(\overline{\rho})).$ This implies that $\overline{\rho}$ extends to $\pi_1(M_\varphi)$ 
with 
$\overline{\rho}(t)$ any element in the centraliser.
The values taken in \Cref{eq:PSL_characters} are $(4, 0, 0, 0)$, $(0, 4, 0, 0)$, $(0, 0, 4, 0)$, $(0, 0, 0, 4)$. Hence, no two of these four extensions are conjugate, but all have image equal to $\im(\overline{\rho}).$ \Cref{exa:A2B-2} shows that the character that lifts depends on $\varphi.$ The four extensions are characterised by $\prho(t) = 1,$ $\prho(t) = \prho(a),$ $\prho(t) = \prho(b),$ and $\prho(t) = \prho(ab)$ respectively. This proves the claim about the order.
We call these representations the \textbf{four extensions of the Klein four group} and denote them respectively by $\prho_0,$  $\prho_1,$  $\prho_2,$  $\prho_{3}$ and their characters by $\overline{\chi}_0,$  $\overline{\chi}_1,$  $\overline{\chi}_2,$  $\overline{\chi}_{3}.$ Note that there is a natural action of the Klein four group on these characters defined by $\kappa_i \cdot \overline{\chi}_0 = \overline{\chi}_i.$

If $b_1(\varphi)=2,$ then $\varphi$ interchanges two of $\overline{\rho}(a)$, $\overline{\rho}(b)$, $\overline{\rho}(ab)$ and fixes the third. Without loss of generality, assume $\varphi(\overline{\rho}(a))=\overline{\rho}(b).$ Then 
a direct calculation shows that $\overline{\rho}$ extends to a representation of $\pi_1(M_\varphi)$ by letting
	\begin{equation}
		\overline{\rho}(t) = \pm\begin{pmatrix}
			    0&\sqrt{i}\\i \sqrt{i}&0
    		\end{pmatrix},
   		 \qquad
    \overline{\rho}(a) = \kappa_1
    \qquad
    \overline{\rho}(b) = \kappa_2 
    \end{equation}
    The remaining extensions are obtained from this by twisting by the elements $\kappa_i.$
The above extension is conjugate to the twisted representation $t \mapsto \kappa_3\overline{\rho}(t)$ via $\kappa_1$, and their values taken in \Cref{eq:PSL_characters} are $(0, 2, 2, 0).$ Similarly, the two representations obtained by twisting $\prho$ by $\kappa_1$ or $\kappa_2$ are conjugate via $\kappa_1,$ and their values taken in \Cref{eq:PSL_characters} are $(2, 0, 0, 2);$ for instance
\[
t \mapsto \kappa_1\overline{\rho}(t) = 
\pm\begin{pmatrix}
			   i\sqrt{i}&0\\0&-\sqrt{i}
    		\end{pmatrix}
\]
Hence we obtain two non-conjugate extensions and the order of the meridian is as claimed.

If $b_1(\varphi)=1,$ then any extension lifts to $\SL(2, \mathbb{C})$ and hence any two extensions are conjugate. 
\end{proof}

Let $\overline{\rho}\co \pi_1(S) \to \PSL(2, \mathbb{C})$ be an irreducible representation.
We have $C(\im(\overline{\rho}))\cong \mathbb{Z}_2$ if and only if $0 \neq \chi_{\prho}\in\overline{L}_i,$ where 
\begin{align}
\overline{L}_1 &= \{ (x^2, 0, 0, 0 ) \mid x \in \mathbb{C} \} \subset \overline{X}(S)\\
\overline{L}_2 &= \{ (0, y^2, 0, 0 ) \mid y \in \mathbb{C} \} \subset \overline{X}(S)\\
\overline{L}_3 &= \{ (0, 0, z^2, 0 ) \mid z \in \mathbb{C} \} \subset \overline{X}(S)
\end{align}
are the images of the coordinate axes in $X(S).$ The points $(4, 0, 0, 0 )$, $(0, 4, 0, 0 )$ and $(0, 0, 4, 0 )$ are the only reducible characters on these lines. 
In $\overline{R}(S)$ we may choose for each $\overline{L}_i$ a curve of representations that have constant centraliser, and such that the curves meet in the representation given in \Cref{eq:paramK4}.

Write $(4p^2, 0, 0, 0 ) = (x, 0, 0, 0 )\in \overline{L}_1$ and let $q\in \mathbb{C}$ such that $p^2+q^2=1$. We may then choose
\begin{align}\label{eq:param_of_PL1}
    \overline{\rho}_1(a) = \begin{pmatrix}
    q&-p\\p&q
    \end{pmatrix}\kappa_1,
    \qquad
    \overline{\rho}_1(b) = \kappa_2,
    \qquad\text{and hence}\qquad
    C(\im(\overline{\rho}_1)) = \langle \kappa_1 \rangle
 \end{align}
Write $(0,-4p^2, 0, 0, 0 ) = (0, y, 0, 0 )\in \overline{L}_2$ and let $q\in \mathbb{C}$ such that $p^2-q^2=1$. We may then choose
\begin{align}\label{eq:param_of_PL2}
    \overline{\rho}_2(a) = \kappa_1
    \qquad
    \overline{\rho}_2(b) =  \kappa_2\begin{pmatrix}
    q&p\\p&q
    \end{pmatrix},
    \qquad\text{and hence}\qquad
    C(\im(\overline{\rho}_3)) = \langle \kappa_2 \rangle
 \end{align}
Let $(0, 0, (p+p^{-1})^2, 0 ) = (0, 0, z, 0 )\in \overline{L}_3.$
We may choose
\begin{align}\label{eq:param_of_PL3}
    \overline{\rho}_3(a) = \kappa_1
    \qquad
    \overline{\rho}_3(b) =  \kappa_2\begin{pmatrix}
    ip&0\\0&-ip^{-1}
    \end{pmatrix},
    \qquad\text{and hence}\qquad
    C(\im(\overline{\rho}_3)) = \langle \kappa_3 \rangle
 \end{align}

\begin{lemma}
\label{lem:centraliser_order2}
Suppose $\overline{\rho}\co \pi_1(S) \to \PSL(2, \mathbb{C})$ is irreducible.
If $\chi_{\overline{\rho}} \in \im (\overline{r})$ and $|C(\im(\overline{\rho}))|=2$, then $|\overline{r}^{-1}(\chi_{\overline{\rho}})| = 2$ and $b_1(\varphi)>1.$
Moreover, either no character in $\overline{r}^{-1}(\chi_{\overline{\rho}})$ lifts to $X(M_\varphi)$; or exactly one lifts and $b_1(\varphi)=3;$ or two of them lift and $b_1(\varphi)=2$.
\end{lemma}

\begin{proof}
We have $0 \neq \chi_{\prho}\in\overline{L}_i.$ Without loss of generality, we may assume $\overline{L}_i = \overline{L}_3,$ and $\chi_{\prho}= (0,0,z_0^2, 0).$ Hence
\[  
(\tr\overline{\rho}(a))^2 = (\tr\overline{\rho}(b))^2= 0 \neq (\tr\overline{\rho}(ab))^2
\]
In $L_3,$ this character has two pre-images, $\chi_+ = (0,0,z_0)$ and $\chi_-=(0,0,-z_0).$  Recall the action of $\overline{\varphi}$ on the coordinate axes. If $b_1(\varphi)=1,$ then the three axes are in the same orbit under $\overline{\varphi}$. Hence $z_0\neq 0,$ implies $b_1(\varphi)>1.$ 

We next show that $|\overline{r}^{-1}(\chi_{\overline{\rho}})| = 2.$
Up to conjugation, we have
\begin{equation*}
    \overline{\rho}(a) = \pm\begin{pmatrix}
    0&1\\-1&0
    \end{pmatrix},
    \qquad
    \overline{\rho}(b) = \pm\begin{pmatrix}
    0&p\\-p^{-1}&0
    \end{pmatrix},
    \qquad\text{and hence}\qquad
    C(\im(\overline{\rho})) = \langle \pm\begin{pmatrix}
    i&0\\0&-i
    \end{pmatrix}
    \rangle
 \end{equation*}
where $z_0 = -(p+p^{-1})\notin\{-2,0,2\}$ since $\overline{\rho}$ is reducible if and only if $z_0=\pm 2$, and the centraliser has order four if and only if $z_0= 0$. These two cases correspond to the characters $(x^2, y^2, z^2, xyz)=(0,0,4,0)$ and $(0,0,0,0)$ respectively with respect to \Cref{eq:param_overline{X}(S)}. 
Assume that $\overline{\rho}$ is the restriction of a representation of $M_\varphi$ with $\overline{\rho}(t) = \pm T$ for some
\[T = \begin{pmatrix}m_1&m_2\\m_3&m_4\end{pmatrix} \in \SL(2,\mathbb{C})\]
We know that $|\overline{r}^{-1}(\chi_{\overline{\rho}})|\leq | C(\im(\overline{\rho}))| =2$. Suppose that $|\overline{r}^{-1}(\chi_{\overline{\rho}})| = 1$. Then the two possible extensions of the representation of $S$ to $M_\varphi$ defined by $\overline{\rho}(t) = \pm T$ or $\overline{\rho}(t) = \pm \begin{pmatrix}
    i&0\\0&-i
\end{pmatrix}T$ are conjugate. In particular, their values taken in \Cref{eq:PSL_characters} are identical. This gives a system of $4$ equations. A direct computation shows that $z_0 \notin\{-2,0,2\}$ implies that  $m_1 = m_2 = m_3 = m_4 = 0.$ This is a contradiction. Hence $|\overline{r}^{-1}(\chi_{\overline{\rho}})| = 2.$ 

If $b_1(\varphi)=2,$ then $\overline{\varphi}$ stabilises one axis and permutes the other two axes. Since $z_0\neq 0,$ it preserves $L_3.$ We therefore have either $\overline{\varphi}(0,0, z) = (0,0,z)$ for each $z \in \mathbb{C}$ or $\overline{\varphi}(0,0, z) = (0,0, -z)$  for each $z \in \mathbb{C}.$  In the second case, $\chi_{\prho}$ does not have a preimage in $X_\varphi(S)$ and hence no character in $\overline{r}^{-1}(\chi_{\overline{\rho}})$ lifts to $X(M_\varphi).$
In the first case, the points $(0,0,z_0)$ and $(0,0,-z_0)$ are both in $X_\varphi(S)$. Since $\overline{\varphi}$ stabilises $L_3$ and permutes $L_1$ and $L_2,$ the group $H$ is generated by $a \mapsto -1,$ $b\mapsto -1$ and hence the points $(0,0, \pm z_0)$ are fixed by $H$ and thus in distinct $H$--orbits. Hence their extensions map to the two distinct characters in $\overline{r}^{-1}(\chi_{\overline{\rho}}).$

If $b_1(\varphi)=3,$ we also have either $\overline{\varphi}(0,0, z) = (0,0,z)$ for each $z \in \mathbb{C}$ or $\overline{\varphi}(0,0, z) = (0,0, -z)$  for each $z \in \mathbb{C}.$ As above, in the second case, neither character in $\overline{r}^{-1}(\chi_{\overline{\rho}})$ lifts. In the first case, the points $(0,0,z_0)$ and $(0,0,-z_0)$ are in $X_\varphi(S)$ and in the same $H$--orbit since $H =  H^1(\pi_1(S), \mathbb{Z}_2).$ According to
\Cref{lem:H-orbits or not}, the extensions of these characters are in the same $H^1(\pi_1(M), \mathbb{Z}_2)$--orbit and hence have the same image in $\overline{r}^{-1}(\chi_{\overline{\rho}})$. So exactly one of the characters in $\overline{r}^{-1}(\chi_{\overline{\rho}})$ lifts and the other does not lift.
\end{proof}

The proof of \Cref{lem:centraliser_order2} supplied the missing details for the following:

\begin{corollary}
\label{cor:diagram_detailed}
Suppose $M_\varphi$ is a hyperbolic once-punctured torus bundle. 
We have the following commutative diagram of maps whose degrees are indicated in the diagram:
\[
\xymatrix{
X^{\text{irr}}(M_\varphi) \ar[d]^{2^{{b_1(\varphi)}}:1}_{q_1} \ar[r]^{r}_{2:1} 	& X_\varphi(S) \ar[d]^{2^{{b_1(\varphi)}-1}:1}_{q_2} \ar@{}[r]|-*[@]{\subset} 	& X(S) \ar[d]^{4:1}_{q_2} \ar@{}[r]|-*[@]{\cong} & \mathbb{C}^3\\
\overline{X}^{\text{irr}}(M_\varphi) \ar[r]^{\overline{r}}_{} 		& \overline{X}_\varphi(S)  \ar@{}[r]|-*[@]{\subset}	& \overline{X}(S)	
}		
\]
Here, each of the maps $r,$ $q_1,$ $q_2$ is a branched covering map onto its image. 
The map $q_1$ is the quotient map associated with the action of $H^1(M_\varphi,\mathbb{Z}_2 ) = \langle \varepsilon\rangle \oplus H,$ and  the map $q_2$ is the quotient map associated with the action of $H^1(S,\mathbb{Z}_2 )$ on $X(S).$ 
The map $r$ is the quotient map associated with the action of $\langle \varepsilon \rangle.$ 
The restriction of $q_2$ to $X_\varphi(S)$ is the quotient map associated with the action of $H$ unless $b_1(\varphi)=2$ and 
one of the coordinate axes is contained in $X_\varphi(S).$
\end{corollary}

\begin{proof}
If $b_1(\varphi)=1$ this is \Cref{pro:diagram_rank1} and there is nothing to prove. If $b_1(\varphi)=3,$ we have $H = H^1(S,\mathbb{Z}_2 )$ and there also is nothing to prove. Hence suppose $b_1(\varphi)=2$ and that $\overline{L_i}$ is the unique coordinate axis in $\overline{X}_\varphi(S).$ It follows from \Cref{lem:H-orbits or not} and the fact that all characters of representations with non-trivial centraliser are contained on the coordinate axes that if $L_i$ is not contained in ${X}_\varphi(S)$, then $q_2$ is the quotient map of the $H$--action. If $L_i\subset {X}_\varphi(S),$ then for the same reasons, $q_2$ restricted to $\overline{X_\varphi(S) \setminus L_i}$ is the quotient map of the $H$--action. However, restricted to $L_i,$ it is the quotient map associated with the sign change on $L_i,$ this has degree two.
\end{proof}

\begin{proposition}
\label{pro:degrees_PSLrestriction}
Let $M_\varphi$ be a hyperbolic once-punctured torus bundle. Let $X$ be a Zariski component of $\overline{X}(M_\varphi)$ containing the character of an irreducible representation.
If $\overline{r} \co X \to \overline{X}_\varphi(S)$ does not have degree one, then the degree is two and $b_1(\varphi)=2$ and $X$ is the preimage of one of the lines $\overline{L}_i \subset \overline{X}_\varphi(S).$
\end{proposition}

\begin{proof}
The above classification of irreducible characters with non-trivial centraliser shows that $\overline{r} \co X \to \overline{X}_\varphi(S)$ has degree one unless $\overline{L}_i\subset \overline{X}_\varphi(S)$ and $X = \overline{r}^{-1}(\overline{L}_i)$ for some $i \in \{1, 2,3 \}.$ In this case, the degree is two.

The conclusion for $b_1(\varphi) =1$ now follows from the description of the action on the lines in the proof of \Cref{lem:dimension_bound_order1}, which gives $\overline{L}_i \cap \overline{X}_\varphi(S) = \{ (0,0,0,0)\}$ for each $i.$
Alternatively, this follows from the fact that each representation lifts and from the uniqueness up to sign in \Cref{lem:irreps in fixed point set}.

Suppose $b_1(\varphi) =3$ and that  $X = \overline{r}^{-1}(\overline{L}_i).$ It follows from \Cref{lem:centraliser_order2} that $X \to \overline{L}_i$ is a 2--fold (possibly branched) cover. Since $\overline{L}_i\cong \mathbb{C}$ and hence is simply connected it follows that the map $X \to \overline{L}_i$ must have a branch point. \Cref{lem:centraliser_order2} shows that there are only two potential branch points: the Klein four group or the reducible character on $\overline{L}_i$. It follows from the  explicit description of the four extensions in the proof of \Cref{lem:K4_number_of_lifts} and the description of the covering map in the proof of  \Cref{lem:centraliser_order2} that the 
extensions of the Klein four group contained in $\overline{r}^{-1}(\overline{L}_i)$ are not ramification points. Hence it must be the reducible character on $\overline{L}_i.$ Without loss of generality suppose $i=3$ and use the parameterisation from \Cref{eq:param_of_PL3}:
\begin{equation}
    \overline{\rho}(a) = \pm\begin{pmatrix}
    0&1\\-1&0
    \end{pmatrix},
    \qquad
    \overline{\rho}(b) = \pm\begin{pmatrix}
    0&1\\-1&0
    \end{pmatrix}
    \qquad\text{and let}\qquad
    C = \pm\begin{pmatrix}
    i&0\\0&-i
    \end{pmatrix}
 \end{equation}
Since $\overline{\rho}(a) = \overline{\rho}(b)$ and both have order two, the hypothesis that $b_1(\varphi) =3$ implies that 
$\overline{\rho}(a)$ and $\overline{\rho}(t)$ commute. Hence there are two possibilities for $\overline{\rho}(t)=T_k$:
\begin{equation}
    T_1 = \pm\begin{pmatrix}
    m_1&m_2\\-m_2&m_1
    \end{pmatrix}
    \qquad\text{or}\qquad
    T_2= \pm\begin{pmatrix}
    m_4&m_3\\m_3&-m_4
    \end{pmatrix}
 \end{equation}
where $m_1^2+m_2^2=1$ or $m_3^2+m_4^2=-1.$ 

The representation defined by $\overline{\rho}_1(t)=T_1,$ $\overline{\rho}_1(a) = \overline{\rho}_1(b) = \kappa_1$ satisfies $\tr[\overline{\rho}_1(t), \overline{\rho}_1(a) ]=2$ and hence is reducible. Indeed, this gives a 1--dimensional family of reducible representations with the quadruple in \Cref{eq:PSL_characters} equal to
\[ (4m_1^2, 4(1-m_1^2), 4(1-m_1^2), 4m_1^2)\]
The representation defined by $\overline{\rho}_2(t)=T_2,$ $\overline{\rho}_2(a) = \overline{\rho}_2(b) = \kappa_1$ satisfies $\tr[\overline{\rho}_2(t), \overline{\rho}_2(a) ]=-2$ and hence is irreducible. For any choice of $m^2_3+m^2_4=-1$, the quadruple in \Cref{eq:PSL_characters} equals
\[ (0, 0, 0, 0) \]
Indeed, this has image isomorphic with $\mathbb{Z}_2\oplus \mathbb{Z}_2$ 
and lifts to a binary dihedral representation into $\SL(2, \mathbb{C}).$

The representations $\overline{\rho}_1$ and $\overline{\rho}_2$ are related by twisting by $C$ (and possibly another element of the centraliser of $\overline{\rho}(a)$ that fixes the fixed points of $\overline{\rho}(a)$). In particular, the reducible character on $\overline{L}_i$ has exactly two preimages in $X = \overline{r}^{-1}(\overline{L}_i).$ This shows that there are no ramification points.
\end{proof}

Examples for the case in \Cref{pro:degrees_PSLrestriction} where the degree is two are given in \Cref{sec:infinite_no_lifts}. In these examples, one extension of the Klein four character is on this line, and the other extension is on the canonical component.

In the case of maximal rank, we can give a complete picture of the Zariski components that are pre-images of the lines:

\begin{proposition}
\label{lem:lines_in_PSL}
If $H_1(M_\varphi,\mathbb{Z}_2)\cong \mathbb{Z}_2^3$, then $\overline{X}_\varphi (S)$ contains all three of $\overline{L}_1$, $\overline{L}_2$ and $\overline{L}_3$. Moreover, each line $\overline{L}_i$ has two pre-images $\overline{L}'_i$ and $\overline{L}''_i$ in $\overline{X}(M_\varphi)$ that are pairwise disjoint. There is a permutation $\sigma \in \Sym(4)$ such that the labels can be chosen such that $\chi_{\sigma(i)} \in \overline{L}''_i$ for $i\in \{1, 2, 3\}$, and $\chi_{\sigma(0)} = \overline{L}'_1\cap \overline{L}'_2\cap \overline{L}'_3$ is the only pairwise intersection point of the six preimages and is a character that lifts to $\SL(2,\mathbb{C}).$
Up to passing to the degree two cover $M_{\varphi^2}$, we may assume that $\sigma$ is the identity.
\end{proposition}

\begin{proof}
The fact that $\overline{X}_\varphi (S)$ contains all three of $\overline{L}_1$,$\overline{L}_2$ and $\overline{L}_3$
 follows from the observation in the proof of \Cref{lem:dimension_bound_order1}, noting that the sign is now irrelevant, and the fact that irreducible representations of the fibre always extend. We know from \Cref{lem:dimension_bound_order1} that $X_\varphi (S)$ either contains all three lines $L_1,$ $L_2,$ and $L_3$ or it contains exactly one of them.

First assume that $X_\varphi (S)$ contains all three lines.
Then $X(M_\varphi)$ contains a character that extends the quaternionic group and is the common intersection of the pre-images $L'_1,$ $L'_2,$ and $L'_3$ in $X(M_\varphi)$ of the three lines. Under the map to the $\PSL(2,\mathbb{C})$--character variety, these map to Zariski components $\overline{L}'_1,$  $\overline{L}'_2,$ $\overline{L}'_3.$ Hence their common intersection is one of the extensions $\chi_j$ of the Klein four group.

Consider the line $\overline{L}_3$ and the parameterisation given in \Cref{eq:param_of_PL3} that is consistent with \Cref{eq:paramK4}. Each point in $\overline{L}_3$ has precisely two preimages in $\overline{X}(M_\varphi),$ and they are related by the action of $\kappa_3.$ It follows that the preimage $\overline{r}^{-1}(\overline{L}_3)$ either has one or two Zariski components. Since the lifting obstruction is constant on topological components of the character variety and $\overline{r}^{-1}(\overline{L}_3)$ contains two distinct extensions of the Klein four group, namely $\chi_j$ and $\kappa_3\cdot \chi_j$, it follows that $\overline{r}^{-1}(\overline{L}_3)$ is a disjoint union of two components $\overline{L}'_3$ and $\overline{L}''_3$ and we have $\chi_j\in \overline{L}'_3$ and $\kappa_3\cdot \chi_j\in \overline{L}''_3$

The same argument applies to the lines $\overline{L}_1$ and $\overline{L}_2$ where the respective action with respect to the appropriate parameterisations consistent with \Cref{eq:paramK4} is given by
$\kappa_1$ and $\kappa_2.$
It now follows from the action of the Klein four group on the characters $\chi_0,$  $\chi_1,$  $\chi_2,$  $\chi_{3}$ that no two of $\overline{L}''_i$ and $\overline{L}''_k$ have one of these characters in common. This implies that they are pairwise disjoint. 

If $\chi_j\neq \chi_0,$ then the meridian is mapped to an element of order two under the character corresponding to the triple intersection. Hence under the map $\overline{X}(M_\varphi) \to \overline{X}(M_{\varphi^2}),$ the character $\chi_j \in \overline{X}(M_\varphi) $ is mapped to $\chi_0 \in \overline{X}(M_{\varphi^2}),$ and each preimage of one of the lines $\overline{L}_i$ in $\overline{X}(M_\varphi)$ is mapped to the corresponding connected component of the preimage of $\overline{L}_i$ in $\overline{X}(M_{\varphi^2})$ that passes through $\chi_0 \in \overline{X}(M_{\varphi^2}).$ This completes the proof of the lemma in the case where $X_\varphi (S)$ contains all three lines since $\kappa_i \cdot \chi_0 = \chi_i.$

Now suppose that $X_\varphi (S)$ only contains one of the three lines $L_1,$ $L_2,$ and $L_3.$ 
Without loss of generality, we may assume this is $L_1.$ Then in $\overline{X}(M_\varphi)$ there is a component $\overline{L}''_1$ that lifts to $L_1.$ The same argument as above shows that $\overline{r}^{-1}(\overline{L}_1)$ has two disjoint components. Suppose $\chi_j \in \overline{L}'_1.$ Then $\kappa_1\cdot\chi_j \in \overline{L}''_1.$
Since the lifting obstruction is constant on topological components, $\overline{r}^{-1}(\overline{L}_2)$ and $\overline{r}^{-1}(\overline{L}_3)$ do not pass through $\kappa_1\cdot\chi_j.$ Then $\overline{r}^{-1}(\overline{L}_2)$ also does not pass through $\kappa_2\kappa_1 \cdot \chi_j = \kappa_3\cdot \chi_j,$ and $\overline{r}^{-1}(\overline{L}_3)$ does not pass through $\kappa_3\kappa_1 \cdot \chi_j = \kappa_2\cdot \chi_j.$ Hence $\overline{r}^{-1}(\overline{L}_2)$ contains $\chi_j$ and $\kappa_2 \cdot \chi_j;$ and $\overline{r}^{-1}(\overline{L}_3)$ contains $\chi_j$ and $\kappa_3 \cdot \chi_j.$ This proves that each extensions of the Klein four group is contained in at least one of the pre-images of the lines $\overline{L}_i.$ The proof is now completed by the observation in the proof of \Cref{pro:degrees_PSLrestriction} that each $\overline{L}_i$ has preimage in $\overline{X}(M_\varphi)$ consisting of two disjoint Zariski components, each containing one of the two extensions of the Klein four group contained in $\overline{r}^{-1}(\overline{L}_i).$ The last statement now follows as in the previous paragraph.
\end{proof}

\begin{remark}
In the examples we computed explicitly, it is always that case that $\chi_{0} = \overline{L}'_1\cap \overline{L}'_2\cap \overline{L}'_3.$
\end{remark}

\begin{example}
We continue the example of $\varphi = \alpha^2\beta^{-2}$ discussed in \Cref{exa:A2B-2}.  As $H_1(M_\varphi;\mathbb{Z}_2)\cong \mathbb{Z}^3_2$, $\overline{L}_1\cup\overline{L}_2\cup\overline{L}_3\in \overline{X}_\varphi(S)$. 
A direct calculation shows that each line in $\overline{X}_\varphi(S)$  is the preimage of two Zariski components in $\overline{X}(M_\varphi)$ under the map $\overline{r}.$ Consider the preimages $\overline{L}_3',$ $\overline{L}_3''$ of $\overline{L}_3.$ Suppose that the component $L_3'\in X(M_\varphi)$ described in \Cref{exa:A2B-2} is the preimage of the component $\overline{L}_3''$ in $\overline{X}(M_\varphi)$. The values in \Cref{eq:PSL_characters} on $\overline{L}_3''$ are $(0,0,0,4-z^2)$. 
Twisting by the generator of $C(\im(\overline{\rho
}))$, we obtain another component $\overline{L}_3'\subset \overline{X}(M_\varphi)$ where the corresponding values are 
$(4,0,0,z^2)$. 
We may choose labels for the preimages of the other lines such that the values taken in \Cref{eq:PSL_characters} for $\overline{L}_1''$ are $(x^2,(x^2-2)^2,0,0)$, for $\overline{L}_1'$ are $(4-x^2,x^2(4-x^2),0,0)$, for $\overline{L}_2''$ are $(y^2,0,(y^2-2)^2,0)$, and for $\overline{L}_2'$ are $(4-y^2,0,y^2(4-y^2),0)$. Note that $\overline{L}_1'$, $\overline{L}_2'$ and $\overline{L}_3'$ intersect at the character with trivial peripheral holonomy, and each of the other $3$ characters extending the Klein four group is contained in precisely one of $\overline{L}_1''$, $\overline{L}_2''$ and $\overline{L}_3''$.
\end{example}


\subsection{Dimensions and lifting}

As an application of the discussion in the previous section, we have the following:

\begin{proposition}
\label{pro:dimension_bound_PSL}
Let $M_\varphi$ be a hyperbolic once-punctured torus bundle. Then every Zariski component of $\overline{X}(M_\varphi)$ is at most one-dimensional. Moreover, if $\{ \chi_{\prho} \} \subset \overline{X}(M_\varphi)$ is a zero-dimensional component, then this character does not not lift to $\SL(2,\mathbb{C}),$ $b_1(\varphi)\neq 1$, $\prho(\im( \pi_1(T) \to \pi_1(M_\varphi))) \cong \mathbb{Z}_2 \oplus \mathbb{Z}_2,$ and $\prho$ is irreducible with trivial centraliser. Moreover, $\chi_{\prho}$ maps to a character on a one-dimensional component in $\overline{X}(M_{\varphi^2}).$
\end{proposition}

\begin{proof}
The same arguments as in the proof of \Cref{pro:dimension_bound} show that 
every component containing only reducible characters is one-dimensional and of genus zero, and that 
that the dimension of each Zariski component of $\overline{X}(M_\varphi)$ is at most one since $M_\varphi$ does not contain a closed essential surface.
Suppose $X$ is a Zariski component containing the character of an irreducible representation $\prho.$ 
\Cref{pro:dimPSL_Thurston} implies that the dimension of $X$ is at least one unless $\prho(\im( \pi_1(T) \to \pi_1(M_\varphi)))$ is trivial or isomorphic with $\mathbb{Z}_2 \oplus \mathbb{Z}_2.$

If $\prho(\im( \pi_1(T) \to \pi_1(M_\varphi)))$ is trivial, then
$\tr\prho(a^{-1}b^{-1}ab) = \pm 2$ since $a^{-1}b^{-1}ab$ is the longitude. If the trace equals $+2,$ then $\prho(a)$ and $\prho(b)$ have a common fixed point on $P^1(\mathbb{C}).$ But since $\rho(t) = \pm I$ this implies that the representation is reducible. Hence $\tr\prho(a^{-1}b^{-1}ab) = -2.$ In this case, $\prho(\langle a, b \rangle) \cong \mathbb{Z}_2 \oplus \mathbb{Z}_2.$ 
It follows that $H_1(M_\varphi; \mathbb{Z}_2) \cong \mathbb{Z}_2^3,$ and that $\chi_{\prho} = \chi_0$ is the extension of the Klein four group with trivial peripheral holonomy. It now follows from \Cref{lem:lines_in_PSL} that this character is contained on a one-dimensional component.

Hence suppose that $\prho(\im( \pi_1(T) \to \pi_1(M_\varphi))) \cong \mathbb{Z}_2 \oplus \mathbb{Z}_2.$ In this case, $\tr\prho(a^{-1}b^{-1}ab) = 0$ and therefore the representation restricted to $\pi_1(S)$ is irreducible and has trivial centraliser. Since the Klein four group lifts to the quaternionic group, this representation does not lift to $\SL(2,\mathbb{C}).$ In particular, 
$b_1(\varphi)\neq 1.$ Now under the map to $\overline{X}(M_{\varphi^2}),$ $\chi_{\prho}$ maps to a character with the property that the peripheral subgroup has image isomorphic with $\mathbb{Z}_2.$ Hence \Cref{pro:dimPSL_Thurston} implies that the image lies on a one-dimensional component of $\overline{X}(M_{\varphi^2}).$
\end{proof}

\begin{remark}
\label{rem:PSL_dim}
We do not have examples of once-punctured torus bundles with zero-dimensional components in $\overline{X}(M_\varphi)$. The purported examples given in \cite[Theorem 7.6]{baker-character-2013} contradict \Cref{pro:dimPSL_Thurston} and a simple description of one-dimensional components containing them is given in \Cref{exa:BP_dim0is1}.
\end{remark}

Heusener and Porti~\cite[\S4.2]{Heusener-varieties-2004} give examples of hyperbolic once-punctured torus bundles that have arbitrarily many one-dimensional Zariski components in $\overline{X}(M_\varphi)$ that do not lift to $X(M_\varphi).$ All of these components are of genus zero. We remark that all these examples satisfy $b_1(\varphi)=3.$ We give similar examples with $b_1(\varphi)=2$ in \Cref{sec:infinite_no_lifts}, and it was already remarked that if $b_1(\varphi)=1,$ then every representation lifts.

Components containing only reducible characters have an analogous characterisation as given in \Cref{subsec:Reducible characters}. Moreover, the proof of \Cref{pro:intersection_red_irrep} gives the analogous result for the $\PSL(2,\mathbb{C})$--character variety. We state this here for completeness:

\begin{proposition}\label{pro:PSL_intersection_red_irrep}
For every Zariski component $X$ of $\overline{X}^{\text{red}}(M_\varphi)$, the restriction $\overline{r} \co X \to \overline{X}_\varphi(S)$ is constant and $X \cap \overline{X}^{irr}(M_\varphi)\neq \emptyset.$ Each character $\overline{\chi}$ in the intersection is contained on exactly one curve $C$ in $\overline{X}^{\text{irr}}(M_\varphi),$ a smooth point of $X$ and $C$ and the intersection at $\overline{\chi}$ is transverse.
\end{proposition}


\subsection{An algebraic subset}

For computations, it is useful to identify $\overline{X}_\varphi(S)$ with the image of a suitable algebraic subset of $X(S)$ under the four-fold branched covering map $X(S)\to \overline{X}(S).$ Let
\begin{align}\label{eq:signed_FP_sets}
X_\varphi^{1}(S) =\; &\{(x,y,z) \in \mathbb{C}^3 |(\phantom{-}x,-y,-z)=\overline{\varphi}(x,y,z) \}\\
X_\varphi^{2}(S) =\; &\{(x,y,z) \in \mathbb{C}^3 |(-x,\phantom{-}y,-z)=\overline{\varphi}(x,y,z) \} \\
X_\varphi^{3}(S) =\; &\{(x,y,z) \in \mathbb{C}^3 |(-x,-y,\phantom{-}z)=\overline{\varphi}(x,y,z) \}
\end{align}
Note that the intersection of any two of the sets $X_\varphi(S), X_\varphi^1(S), X_\varphi^2(S), X_\varphi^3(S)$ is contained in a coordinate axis. Recall that the natural map $q_2\co X(S)\to \overline{X}(S)$ is the quotient  map of the $H^1(S,\mathbb{Z}_2 )$ action. The definitions imply that
\[
q_2^{-1}( \overline{X}_\varphi(S)) = X_\varphi(S) \cup X_\varphi^{1}(S) \cup X_\varphi^{2}(S) \cup X_\varphi^{3}(S)
\]
For each $(x,y,z)\in q_2^{-1}( \overline{X}_\varphi(S)),$ there are $A, B \in 
\SL(2,\mathbb{C})$ with $\tr A = x, \tr B = y$ and $\tr AB = z$ and satisfying \Cref{eq:PSL_setup_A,eq:PSL_setup_B}, where $\varepsilon_a, \varepsilon_b\in \{\pm 1\}$ are determined by
\[
(\varepsilon_a \;x,\; \varepsilon_b\; y,\; \varepsilon_a\varepsilon_b\; z)=\overline{\varphi}(x,y,z)
\]
The above sets therefore give us a simple way to compute the representations of $\pi_1(M_\varphi)$ into $\PSL(2,\mathbb{C})$. 
We now determine redundancies that arise from the action of the elements of $H^1(S,\mathbb{Z}_2 )$ that are not in $H.$ 

View $h \in H^1(S,\mathbb{Z}_2 )$ as a homomorphism 
$h\co \pi_1(S) \to \{ \pm 1\}.$
Then $h \in H$ if and only if $h(a) = h(\varphi(a))$ and $h(b) = h(\varphi(b)).$
Let $\overline{\rho}\co \pi_1(M_\varphi) \to \PSL(2,\mathbb{C})$ and use the set up from \Cref{eq:PSL_setup_A,eq:PSL_setup_B}. Then $\prho$ lifts to $\SL(2,\mathbb{C})$ if and only if there is $h \in H^1(S,\mathbb{Z}_2 )$ with
\begin{align*}
h(a) &=  h(\varphi(a))\; \varepsilon_A\\
h(b) &= h(\varphi(b))\; \varepsilon_B
\end{align*}
Note that the signs $\varepsilon_A$ and $\varepsilon_B$ are uniquely determined by $\prho$ if $H = H^1(S,\mathbb{Z}_2 )$ and otherwise they depend on the choice of the matrixes $A$ and $B.$
The action of $h\in H^1(S,\mathbb{Z}_2 )$ is given by:
\[
h\cdot(x,y,z) \mapsto (\;h(a)x,\; h(b)y,\; h(ab)z\;) 
\]
This induces a permutation of the sets $X_\varphi(S)$, $X_\varphi^{1}(S),$ $X_\varphi^{2}(S),$ and $X_\varphi^{3}(S).$
The permutation of the sets is determined by the action on the defining equations:
\[
\overline{\varphi}(\;x,\;y,\;z) = (\;h(a)h(\varphi(a))\varepsilon_A \;x, \;h(b)h(\varphi(b))\varepsilon_B\; y, \;h(ab)h(\varphi(ab))\varepsilon_A\varepsilon_B\; z\;)
\]
Each set is stabilised by $H \le H^1(S,\mathbb{Z}_2 )$ and the complementary elements permute the sets $X_\varphi(S),$ $X_\varphi^1(S),$ $X_\varphi^2(S),$ $X_\varphi^3(S).$ 

If $b_1(\varphi)=3,$ then $H = H^1(S,\mathbb{Z}_2 )$ and hence each of the sets is fixed under this action. In particular, 
$\overline{\rho}$ does not lift if at least one of $\varepsilon_A$ and $\varepsilon_B$ equals $-1.$ 
Define
\[
X_\varphi^\perp (S) = X_\varphi^{1}(S) \cup X_\varphi^{2}(S) \cup X_\varphi^{3}(S)
\]
If $b_1(\varphi)=2,$ then the action of $H^1(S,\mathbb{Z}_2 )$ has two orbits, each containing two sets. Let
\[
X_\varphi^\perp (S) = X_\varphi^{i}(S) 
\]
where $X_\varphi^{i}(S)$ is one of the sets not in the orbit of $X_\varphi(S).$

If $b_1(\varphi)=1,$ then there is just one orbit and we define $X_\varphi^\perp (S) = \emptyset.$ This observation can be viewed as an elementary proof of the fact that every representation into $\PSL(2, \mathbb{C})$ lifts to $\SL(2, \mathbb{C})$ in this case.

These definitions imply that the map 
\[
X_\varphi (S) \cup X_\varphi^\perp (S) \to \overline{X}_\varphi(S)
\]
is surjective, and corresponds to the quotient map of the action of $H^1(S,\mathbb{Z}_2 )$ twisted via $\varphi.$



\section{Examples}
\label{sec:examples}

This section provides details about the three infinite families of once-punctured torus bundles mentioned in the introduction, as well as the beginnings of a census. We begin by stating some general results that will be used repeatedly. The main ingredients in the proofs are properties of a sequence of Fibonacci polynomials, and determining irreducibility and genus of a plane algebraic curve from the Newton polygon of a defining polynomial.


\subsection{The Fibonacci polynomials}
\label{sec:fibonacci}
This section collects some facts about a family of recursive polynomials, which are used throughout the computation in \cite{baker-character-2013}. The Fibonacci polynomials will be used to compute $X_{\varphi}(S)$ for our examples.
\begin{definition}
For every integer $n$, the $n$-th Fibonacci polynomial $f_n(u)$ is defined by the recursive relation
$$f_n(u) = uf_{n-1}(u)-f_{n-2}(u)$$
where $f_0(u) = 0$ and $f_1(u) = 1$.
\end{definition}
A number of useful properties of $f_n$ according to \cite{baker-character-2013} are listed below. 
\begin{lemma}
\label{lem: fibonacci poly}
\begin{enumerate}
    \item If $u = s+s^{-1}$, then 
    \begin{equation*}
        f_n(u) = \begin{cases}
        \dfrac{s^n-s^{-n}}{s-s^{-1}}\text{, if }u \neq \pm 2\\
        n\text{, \hspace{1cm}if } u = 2\\
        (-1)^{n+1}n\text{, if } u = -2
        \end{cases}
    \end{equation*}
    \item If $n\neq 0$, the degree of $f_n(u)$ is $|n|-1$.
    \item $f_n(u)$ is divisible by $u$ if and only if $n$ is even. If $n\neq 0$, $f_n(u)$ is not divisible by $u^2$.
    \item For any integer $n$, $f_n(u)$, $f_{n+1}(u) - f_{n}(u)$, $f_{n+1}(u) + f_{n}(u)$ and $f_{n+2}(u) - f_{n}(u)$ are separable except for $f_0(u) = 0$.
    \item $f_{n+2}(u) - f_{n+1}(u) = 0$ and $f_{n+1}(u) - f_{n}(u) = 0$ do not have a common root.
\end{enumerate}
\begin{proof}
Part (1) follows from the defining relation. Parts (2), (3) and (4) follow from \cite[Lemmas 4.3, 4.4 and 4.11]{baker-character-2013}. A direct calculation using (1) results in (5).
\end{proof}
\end{lemma}


\subsection{Genus and the Newton Polygon}
\label{sec:Genus and the Newton Polygon}

A classical link between the Newton polygon and the genus of an irreducible algebraic curve is known as Baker's formula:
\begin{theorem}[Baker~\cite{Baker-examples-1893}]
Suppose that $F(x,y)=0$ defines an irreducible algebraic curve $X$ in $\mathbb{C}^2$. The genus of $X$ is at most the number of lattice points in the interior of the Newton polygon of $F.$
\end{theorem}
Khovanski\u{\i}~\cite{Hovanskii-genus-1978} showed that one generically has equality instead of an upper bound. We use a version of this result as implied by Beelen and Pellikaan~\cite{Beelen_Pellikaan_NetwonPolygon}. A polynomial $F(x,y)\in \mathbb{C}[x,y]$ is said to be \textbf{nondegenerate} with respect to its Newton polygon if for every edge $\gamma$ of its Newton Polygon with the corresponding polynomial $F_\gamma$, the ideal generated by $F_\gamma$, $x\frac{\partial F_\gamma}{\partial x}$ and $y\frac{\partial F_\gamma}{\partial y}$ has no zero in $(\mathbb{C}/\{0\})^2$.
\begin{corollary}[Beelen-Pellikaan]
\label{lem:Newton polygon}
Suppose that $F(x,y)=0$ defines an irreducible algebraic curve $X$ in $\mathbb{C}^2$. If $F$ is nondegenerate with respect to its Newton polygon and it is smooth at every point $(x,y)\in X$ where $xy\neq 0$, then the genus of its nonsingular model is equal to the number of lattice points in the interior of its Newton polygon.
\end{corollary}

\begin{proof}
We show that this is implied by \cite[Theorem 4.2]{Beelen_Pellikaan_NetwonPolygon}. The only statement that requires proof is that the smoothness hypothesis implies that the singular points of the homogeneous curve with equation $F^*(x, y, z) = 0$ are among $(0 : 0 : 1),$ $(0 : 1 : 0)$ and $(1 : 0 : 0),$ where $F^*$ denotes the homogenisation of $F.$ If $S$ is a singular point other than the three points, it is not in the line of infinity and corresponds to a singular point $s=(x,y)\in X$ such that $xy\neq 0$. This contradicts the smoothness condition.
\end{proof}

We will also appeal to the following result \cite[Lemma 5.1]{Castryck-nondegeneracy-2009}. The authors thank Michael Joswig for pointing them to this reference.

\begin{lemma}[Castryck-Voight~\cite{Castryck-nondegeneracy-2009}]
\label{lem:hyperelliptic_newton}
Suppose that $F(x,y)=0$ defines an irreducible algebraic curve $X$ in $\mathbb{C}^2$. Assume that $F$ is nondegenerate with respect to its Newton polygon, and that there are at least two lattice points in the interior of the Newton polygon of $F.$ Then $X$ is hyperelliptic if and only if the interior lattice points of the Newton polygon are collinear.
\end{lemma}


\subsection{The family $M_n$}
\label{sec:first_family}
The genera of the Zariski components of the $\SL(2,\mathbb{C})$-- and $\PSL(2,\mathbb{C})$--character varieties of the infinite family of once-punctured torus bundles $M_n$ with monodromies $\varphi_n = AB^{n+2}$ were determined by Baker and Petersen~\cite[Theorem 5.1]{baker-character-2013} via a birational isomorphism between $X^{irr}(M_{\varphi_n})$ and a family of hyperelliptic curves. In this section, we compute the corresponding varieties $X_{\varphi_{n}}(S)$. Note that $\tr (\varphi_n)=-n$ and $M_n$ is hyperbolic if $n\geq 3$ or $n\leq-3$. For computational simplicity, we only consider the case when $n\geq 3$ and $n$ is odd, and so $b_1({\varphi_n})=1.$

The corresponding family of framings $\varphi_n$ admit the following form:
\begin{equation*}
    {\varphi_n} = \begin{cases}
    a\rightarrow a(a^{-1}b^{-1})^{n+2}\\
    b\rightarrow ba
    \end{cases}
\end{equation*}

By a direct computation as in \Cref{lem:im(r)}, the binary dihedral characters for $X(M_n)$ are as follows. 
\begin{fact}
For every odd integer $n\geq 3$, the ramification points of $r$ are $$(\beta^{-2k}+\beta^{2k},\beta^k+\beta^{-k},\beta^k+\beta^{-k},0,0,0,0)$$ where $\beta = e^{\frac{2\pi i}{n-2}}$, $\beta^{2k}\neq 1$ and $k = 0,1,...,n-3$. There are 
$\frac{n-3}{2} 
= \frac{1}{2}|n-2| - \frac{1}{2}
= \frac{1}{2}|\tr\varphi_n+2| - 2^{{b_1({\varphi_n})-2}}
$ such binary dihedral characters in total, which agrees with 
\Cref{eq:binary_count}.
\end{fact}

\subsubsection{Topology of $X_{\varphi_n}(S)$} 
In this subsection, we compute the fixed point set $X_{\varphi_n}(S)$. Substituting the automorphism $\varphi_n$ into the definition of $X_{\varphi_n}(S)$ in \Cref{sec:Restriction}, we obtain the following 3 equations
\begin{align*}
    \tr\rho(a) &= \tr \rho(a(a^{-1}b^{-1})^{n+2})\\
    \tr\rho(b) &= \tr \rho(ba)\\
    \tr\rho(ab) &= \tr \rho(a(a^{-1}b^{-1})^{n+1})
\end{align*}
The second equation gives $y=z$. Define polynomials $P_n(x,y,z) = \tr \rho(a(a^{-1}b^{-1})^n)$ for $n\in\mathbb{Z}$.
Using the trace identities in \Cref{sec:Preliminaries}, $P_n$ satisfies the recursive relation  $P_n(x,y,z) = zP_{n-1}(x,y,z)-P_{n-2}(x,y,z)$ with $P_0(x,y,z) = x$ and $P_1(x,y,z) = y$. Then $X_{\varphi_n}(S)$ is generated by $P_{n+2}-x,P_{n+1}-z$ and $y-z$. Recall the Fibonacci polynomial $f_n$ in \Cref{sec:fibonacci}. We observe that 
\begin{equation}\label{eq: Pn}
    P_n(x,y,z) = y f_n(z) - x f_{n-1}(z)\text{, for }n\in\mathbb{Z}
\end{equation}
To simplify the defining equations of $X_{\varphi_n}(S)$, we use the idea of Buchberger's algorithm~\cite{Cox-ideal-2015} which is used to compute the Groebner basis of an ideal of a polynomial ring. Denote $X_{n+2} = P_{n+2}-x = P_{n+2}-P_{0}$, $X_{n+1}=P_{n+1}-P_{1}$. We compute the S-polynomial of the first two polynomials recursively to summarise $X_{n+2}$ and $X_{n+1}$ into one single polynomial.
\begin{lemma}\label{lem: S-polynomials}
Define $X_{i+2} = zX_{i+1}-X_{i}$ for $i\geq 0$. When $n=2k+1$ is a positive odd integer, $\langle X_{n+2},X_{n+1}\rangle = \langle P_{k+2}-P_{k+1}\rangle$. When $n=2k$ is a positive even integer, $\langle X_{n+2},X_{n+1}\rangle = \langle P_{k+2}-P_{k}\rangle$.
\end{lemma}
\begin{proof}
Using induction, we can prove that $X_{i} = P_{i}-P_{n+2-i}$ for $0\leq i\leq n+2$. Since $X_{i+2}$ is a linear combination of $X_{i+1}$ and $X_{i}$, $\langle X_{i+2}$, $X_{i+1}\rangle \subseteq \langle X_{i+1},X_{i}\rangle$. By definition, $\langle X_{i+1}$, $X_{i}\rangle \subseteq \langle X_{i+2},X_{i+1}\rangle$. Thus, $\langle X_{i+2}$, $X_{i+1}\rangle = \langle X_{i+1},X_{i}\rangle$.

Hence when $n=2k+1$,
\begin{align*}
    \langle X_{n+2}, X_{n+1}\rangle &=...=\langle X_{k+2},X_{k+1}\rangle = \langle P_{k+2}-P_{k+1} \rangle
\end{align*}
The even case follows analogously.
\end{proof}

When $n$ is odd, $X_{\varphi_n}(S) = V(\langle P_{k+2}-P_{k+1},y-z\rangle)$. Under the restriction $(x,y,z)\mapsto (x,y)$, $X_{\varphi_n}(S)$ is homeomorphic to its image, denoted by $U_n$. Note that $U_n = V(\langle p_{k+2}-p_{k+1} \rangle)$ where $p_n(x,y) = P_n(x,y,y)$. Combining this with \Cref{eq: Pn}, $U_n$ is generated by \[g_k(x,y) \coloneqq p_{k+2}-p_{k+1} = y(f_{k+2}(y)-f_{k+1}(y))-x(f_{k+1}(y)-f_{k}(y))\] By \Cref{lem: fibonacci poly}, the Newton polygon of $g_k$ is shown in \Cref{fig: newton polygon Mn}. 
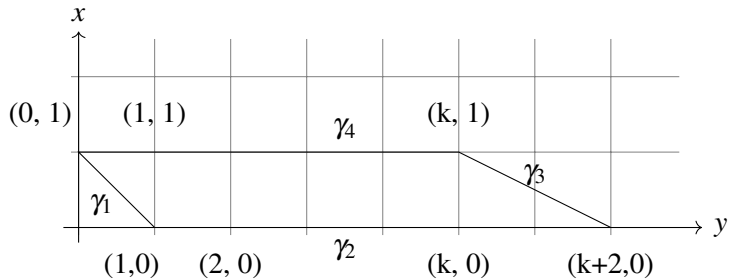
\begin{figure}[htbp]
    \centering
    \begin{tikzpicture}
    \draw[very thin,color=gray] (-0.1,-0.1) grid (7.9,2.5);    
    \draw[->] (-0.2,0) -- (8.2,0) node[right] {$y$};
    \draw[->] (0,-0.2) -- (0,2.6) node[above] {$x$};

    \draw(1,0)--(0,1)--(5,1)--(7,0);
    \node at (0.7,-0.5) {(1,0)};
    \node at (-0.5,1.5) {(0, 1)};
    \node at (2,-0.5) {(2, 0)};
    \node at (1,1.5) {(1, 1)};    
    \node at (5,-0.5) {(k, 0)};
    \node at (5,1.5) {(k, 1)};
    \node at (7,-0.5) {(k+2,0)};
    \node at (0.3,0.3) {$\gamma_1$};
    \node at (3.5,-0.3) {$\gamma_2$};
    \node at (3.5,1.3) {$\gamma_4$};
    \node at (6,0.7) {$\gamma_3$};
\end{tikzpicture}
    \caption{Newton polygon of $g_k$}
    \label{fig: newton polygon Mn}
\end{figure}
\begin{lemma}\label{lem:irreducible Mn}
$g_k(x,y)$ is irreducible in $\mathbb{C}[x,y]$ for any positive integer $k$.
\end{lemma}
\begin{proof}
If $g_k(x,y)$ is reducible, then we can factorise it into $g_k(x,y) = s_k(x,y)t_k(x,y)$, where $t_k,s_k$ are non-constant polynomials in $\mathbb{C}[x,y]$. Since $g_k(x,y)$ is linear in $x$, without loss of generality, we assume that $t_k(x,y) = t_k(y)$ and $s_k(x,y) = x \;s^{(1)}_k(y) + s^{(2)}_k(y)$, where $s^{(i)}_k(y)$ and $t_k(y)$ are polynomial in $\mathbb{C}[y]$. Then
\begin{equation*}
\begin{cases}
    y(f_{k+2}(y) - f_{k+1}(y)) = s^{(2)}_k(y)t_k(y)\\
            -(f_{k+1}(y)-f_{k}(y)) = s^{(1)}_k(y)t_k(y)
\end{cases}
\end{equation*}
Since $t_k(y)$ is non-constant, by the fundamental theorem of algebra, $t_n(y)=0$ has a solution $\alpha\in\mathbb{C}$. \Cref{lem: fibonacci poly}  tells us that $f_{k+1}(y)-f_{k}(y)$ has a nonzero constant term, which implies that $\alpha\neq 0$. Then $\alpha$ is a common zero of $f_{k+2}(y) - f_{k+1}(y)$ and $f_{k+1}(y) - f_{k}(y)$. This contradicts \Cref{lem: fibonacci poly}~(5). Hence $g_k(x,y)$ is irreducible.
\end{proof}
\begin{lemma}\label{lem:singularity Un}
For $n=2k+1\geq 3$, there is no singular point in $U_n=V(g_{k})$ and $g_{k}$ is nondegenerate.
\end{lemma}
\begin{proof}
Suppose $(x,y)$ is a singular point. Then
\begin{equation*}
    \begin{cases}
        y(f_{k+2}(y)-f_{k+1}(y))-x(f_{k+1}(y)-f_{k}(y)) = 0 \\
        \frac{\partial g_k}{\partial x}(x,y) = f_{k+1}(y)-f_{k}(y) = 0 
    \end{cases}
\end{equation*}
This clearly contradicts \Cref{lem: fibonacci poly}~(5). The nondegeneracy condition for edges $\gamma_1$ and $\gamma_3$ in \Cref{fig: newton polygon Mn} is obvious. For the edge $\gamma_2$, $F_{\gamma_2}(x,y) = y(f_{k+2}(y)-f_{k+1}(y))$. Suppose that $(x,y)\in(\mathbb{C}/\{0\})^2$ is a common zero of $F_{\gamma_2}$, $x\frac{\partial F_{\gamma_2}}{\partial x}$ and $y\frac{\partial F_{\gamma_2}}{\partial y}$. This results in 
$$\begin{cases}
    f_{k+2}(y)-f_{k+1}(y) = 0\\
    f_{k+2}'(y)-f_{k+1}'(y) = 0
\end{cases}$$
There is no such $y\neq 0$ satisfying the above two equations by \Cref{lem: fibonacci poly}~(1). In detail, it is clear that $y=\pm 2$ are not zeros. When $y \neq \pm 2$, the first equation implies that $p^{2k+3}+1=0$ and $p\neq -1$ given $y=p+p^{-1}$. Using change of variable formula $\frac{df_n}{dp} = \frac{df_n}{dy}\frac{dy}{dp}$, the second equation implies that 
$$\dfrac{p^{-k}(k(p+1)(p^{2k+3}-1)-1-2p+2p^{2k+3}+p^{2k+4})}{(p^{2}-1)(1+p)^2} = 0$$
Since $y\neq 0$,
$$k(p+1)(p^{2k+3}-1)-1-2p+2p^{2k+3}+p^{2k+4} = 0$$
Combining the above equation with $p^{2k+3}=-1$ implies $p=-1$, which is a contradiction. The nondegeneracy condition for edge $\gamma_4$ can be checked using a similar argument.
\end{proof}
\begin{proposition}\label{pro:genus Mn}
If $n=2k+1\geq3$, then
the variety $\overline{X}^{\text{irr}}(M_\varphi)\cong X_{\varphi_n}(S)$ is irreducible and a curve of genus 0.
\begin{proof}
By \Cref{lem:irreducible Mn} and \Cref{lem:singularity Un}, $g_{k}(x,y)$ satisfies  assumptions in \Cref{lem:Newton polygon}, which implies that the genus of $U_n$ is $0$. As $U_n\cong X_{\varphi_n}(S),$ this proves the statement by \Cref{cor:iso for b1=1}.
\end{proof}
\end{proposition}

According to the Newton polygon~\Cref{fig: newton polygon Mn} and the theory of Puiseux expansions, there are $k$ ideal points such that $x\rightarrow\infty, y\rightarrow c$ (a nonzero constant) and one ideal point such that $x\rightarrow\infty,y\rightarrow\infty$. Hence there are $k+1 = \frac{n+1}{2}$ ideal points. Since the number $\frac{n-3}{2}$ of branch points is positive if $n>3,$ the variety $X^{\text{irr}}(M_n)$ is also irreducible, and the same can be confirmed by direct calculation if $n=3.$ Hence
\Cref{cor:genus_bounds} implies:
\begin{theorem}
Suppose that $g_n$ denote the genus of $X^{\text{irr}}(M_n)$. When $n$ is a positive odd integer and $n\geq 3$, $g_n$ is bounded by
    $$\frac{n-7}{4} \leq g_n \leq \frac{n-3}{2}$$
\end{theorem}

Baker and Petersen~\cite[Theorem 5.1]{baker-character-2013} show that $g_n = \frac{n-3}{2}$ for $n$ odd. So there are $n-1$ branch points in total. Hence in addition to the $\frac{n-3}{2}$ branch points at the binary dihedral characters in $X(M_\varphi(\lambda))\subset X(M_\varphi)$, there are $\frac{n+1}{2}$ branch points at ideal points. In particular, every ideal point is a branch point (and in the case $n=3$ these are the only branch points).


\subsection{The family $N_n$}
\label{sec:second_family}

Let $N_n$ be the once-punctured torus bundle with monodromy $\psi_n = AB^{n+2}A$. Note that $\tr (\psi_n)=-2n-2$ and $N_n$ is hyperbolic when $n\geq 1$ or $n\leq -3$. The corresponding family of framings $\psi_n$ admit the following form
\begin{equation*}
    {\psi_n} = \begin{cases}
    a\rightarrow a(a^{-1}b^{-1})^{n+2}\\
    b\rightarrow ba^2(a^{-1}b^{-1})^{n+2}
    \end{cases}
\end{equation*}
We restrict to $n=2k+1\ge 1$ odd and hence $b_1(\psi_n)=2.$

The corresponding fixed-point set $X_{\psi_n}(S)$ is generated by the following equations
\begin{align*}
    \tr\rho(a) &= \tr \rho(a(a^{-1}b^{-1})^{n+2})\\
    \tr\rho(b) &= \tr \rho(ba^2(a^{-1}b^{-1})^{n+2})\\
    \tr\rho(ab) &= \tr \rho(a(a^{-1}b^{-1})^{n+2}ba^2(a^{-1}b^{-1})^{n+2})
\end{align*}
Using the trace identities in \Cref{sec:Preliminaries},
the above equations can be simplified to $x=P_{n+2}(x,y,z)$, $y=P_{n+1}(x,y,z)$ and $z =xy-z$, where $P_n$ is defined by \Cref{eq: Pn}. Then $X_{\psi_n}(S) = V(\langle P_{n+2}-x,P_{n+1}-y,xy-2z\rangle)$. By \Cref{lem: S-polynomials}, $X_{\psi_n}(S) = V(\langle P_{k+2}-P_{k+1},xy-2z\rangle)$ when $n=2k+1$.

If $x=0,$ then $z=0$ by second equation, and then $0=P_{k+2}-P_{k+1}=y (f_{k+2}(z) -f_{k+1}(z))$ which implies $y=0$ by \Cref{lem: fibonacci poly}~(3). In particular, there is a unique point in $X_{\psi_n}(S)$ where $x=0.$

When $x\neq0$, we substitute $y=\frac{2z}{x}$ into $P_{k+2}(x,y,z)-P_{k+1}(x,y,z)=0$ and multiply by $x$ on both sides, giving
\begin{equation*}
    h_k(x,z) \coloneqq 2(f_{k+2}(z)-f_{k+1}(z))z-(f_{k+1}(z)-f_{k}(z))x^2=0
\end{equation*}

Define the variety $V_n$ as the variety in $\mathbb{C}^2$ generated by $h_k\in\mathbb{C}[x,z]$. Then $X_{\varphi_n} S$ and $V_n$ are birational via rational maps 
\begin{align*}
&
\begin{aligned}
r_1:X_{\psi_n} S&\dashrightarrow V_n\\
    (x,y,z)&\rightarrow(x,z)
\end{aligned}
&
\begin{aligned}
 r_2:V_n&\dashrightarrow X_{\psi_n} S\\
    (x,z)&\rightarrow(x,\frac{2z}{x},z)
\end{aligned}
\end{align*}

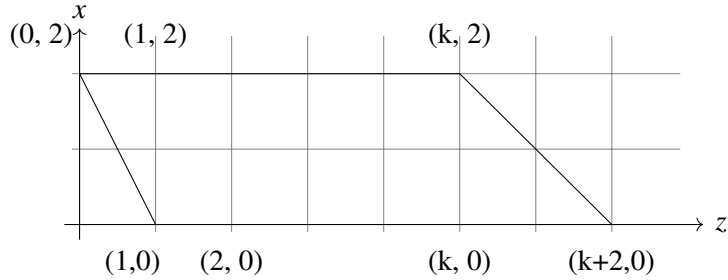
\begin{figure}[htbp]
    \centering
    \begin{tikzpicture}
    \draw[very thin,color=gray] (-0.1,-0.1) grid (7.9,2.5);    
    \draw[->] (-0.2,0) -- (8.2,0) node[right] {$z$};
    \draw[->] (0,-0.2) -- (0,2.6) node[above] {$x$};

    \draw(1,0)--(0,2)--(5,2)--(7,0);
    \node at (0.7,-0.5) {(1,0)};
    \node at (-0.5,2.5) {(0, 2)};
    \node at (2,-0.5) {(2, 0)};
    \node at (1,2.5) {(1, 2)};    
    \node at (5,-0.5) {(k, 0)};
    \node at (5,2.5) {(k, 2)};
    \node at (7,-0.5) {(k+2,0)};
\end{tikzpicture}
    \caption{Newton polygon of $h_k(x,z)$}
    \label{fig: Newton polygon Nn}
\end{figure}

\begin{lemma}\label{lem: irreducible Nn}
$h_k(x,z)$ is an irreducible polynomial in $\mathbb{C}[x,z]$ for every positive integer $k$.
\end{lemma}
\begin{proof}
If $h_k(x,z)$ is reducible, then we can factorise it into $h_k(x,z) = s_k(x,z)t_k(x,z)$, where $t_k,s_k$ are non-constant polynomials in $\mathbb{C}[x,z]$. Then the Newton polygon of $h_k(x,z)$ is the Minkowski sum of the Newton polygons of $s_k(x,z)$ and $t_k(x,z)$. The Newton polygon of $h_n(x,z)$ is shown in \Cref{fig: Newton polygon Nn}. Travelling along the boundary of the Newton polygon, we obtain the boundary vector sequence
$$v(g_k) = \bigg( (1,-2),(1,0),...,(1,0),(-1,1),(-1,1),(-1,0),...,(-1,0) \bigg)$$ where the number of repetitions of $(1,0)$ and $(-1,0)$ in $v(h_k)$ are $k+1$ and $k$ respectively.
The sequence $v(h_k)$ can be partitioned into two disjoint nonempty subsequences $v(s_k)$ and $v(t_k)$, each of which sums to zero. Given $(1,-2)$, $(-1,1)$ and $(-1,1)$ are the only three vectors with nonzero second components, they must be in the same sequence. Without loss of generality, we assume that they are in $v(s_k)$. Noting that $h_k$ has no linear term in $x$, we may write $t_k(x,z) = t_k(z)$ and $s_k(x,z) = x^2s^{(1)}_k(z) + s^{(2)}_k(z)$, where $s^{(i)}_k(z)$ and $t_k(z)$ are polynomials in $\mathbb{C}[z]$. 
Hence \begin{equation*}
\begin{cases}
    2(f_{k+2}(z) - f_{k+1}(z))z = s^{(2)}_k(z)t_k(z)\\
            -(f_{k+1}(z)-f_{k}(z)) = s^{(1)}_k(z)t_k(z)
\end{cases}
\end{equation*}
The remainder of the proof is analogous to the proof of \Cref{lem:irreducible Mn}, where one obtains a contradiction to \Cref{lem: fibonacci poly}~(5).
\end{proof}

\begin{lemma}\label{lem: nondegeneracy Nn}
When $n=2k+1\geq 1$, $h_k$ is nondegenerate and there is no singular point in $V_n=V(h_k)$.
\begin{proof}
The proof is almost identical to \Cref{lem:singularity Un}.
\end{proof}
\end{lemma}
Combining \Cref{lem: irreducible Nn,lem: nondegeneracy Nn} with 
\Cref{lem:Newton polygon} and \Cref{lem:hyperelliptic_newton} implies:
\begin{lemma}
The curve $V_n\cong X_\varphi(S)$ is hyperelliptic of genus $k$ for each $n=2k+1\ge 3$.
\end{lemma}

\subsubsection{$\PSL(2,\mathbb{C})$--character variety}
\label{sec:infinite_no_lifts}

Since $n$ is odd, $b_1(\psi_n) = 2$. If a representation $\overline{\rho}\in\overline{R}(N_n)$ lifts to $\SL(2,\mathbb{C})$, then it generically has $4$ lifts. Let $\overline{X}_0(N_n)$ be the subvariety of $\overline{X}^{\text{irr}}(N_n)$ consisting of the characters of all $\PSL(2,\mathbb{C})$--representations that lift to representations into $\SL(2,\mathbb{C})$. By \Cref{cor:diagram}, $\overline{X_0}(N_n)\cong q_2({X}_{\psi_n}(S))$.
We have the following proposition. 
\begin{proposition}
For every positive odd integer $n=2k+1$, the subvariety $\overline{X}_0(N_n)$ is birational to the affine line.
\end{proposition}

\begin{proof}
Since $\overline{X}_0(N_n)$ is birational to $q_2({X}_{\varphi_n}(S))$, it suffices to show that the latter has genus $0$.

Take a homomorphism $h\in \Hom(\pi_1 (N_n),\mathbb{Z}_2)$. A direct calculation shows that either $(h(a),h(b)) = (1,1)$ or $(h(a),h(b)) = (-1,-1)$ since $n$ is odd. Hence $q_2({X}_{\varphi_n}(S))$ is the quotient of ${X}_{\varphi_n}(S)$ by the involution $(x,y,z)\rightarrow(-x,-y,z).$

Recall that $X_{\psi_n}(S)$ is birational to its restriction $V_n$. Let $\overline{V}_n$ be the variety corresponding to the identification of $V_n$ under the involution $(x,z)\rightarrow(-x,z)$.  Letting $(X,Z) = (x^2,z)$, $\overline{V}_n$ is generated by \begin{equation*}
    \overline{h}_k(X,Z) \coloneqq 2(f_{k+2}(Z)-f_{k+1}(Z))Z-(f_{k+1}(Z)-f_{k}(Z))X
\end{equation*}
The genus of $\overline{V}_n$ is seen to be $0$ by a simple modification of the argument of \Cref{pro:genus Mn}.
\end{proof}

For each irreducible $\PSL(2,\mathbb{C})$--representation $\overline{\rho}$ which does not lift, we may choose $T, A, B\in \SL(2,\mathbb{C})$ with the properties $\overline{\rho}(t) = \{\pm T\},$ $\overline{\rho}(a) = \{\pm A\},$ $\overline{\rho}(b) = \{\pm B\}$ and
\begin{align*}
T^{-1}AT&=- \psi_n(A) \\
T^{-1}BT&= \psi_n(B)
\end{align*}
Using the convention from \Cref{eq:signed_FP_sets}, the traces of these representatives under the restriction map $r$ are contained in the set:
\begin{align*}
    X_{\psi_n}^{2} (S) &= \{(x,y,z) \in \mathbb{C}^3 |(-x,y,-z)=\overline{\psi_n}(x,y,z) \}\\
   &= V(P_{n+2}+x,P_{n+1}-y, xy)
\end{align*}
This implies that the components of $X_{\psi_n}^{2} (S)$ are:
\begin{align}
    \label{eq:coord_line_psi_n}&L_3 = \{(0,0,z)\mid z\in\mathbb{C}\}\\
    \label{eq:alpha_line_psi_n}&C_\xi = \{(0,y,\xi)\mid y\in\mathbb{C}\}\text{ where }\xi
    \text{ is a root of }f_{k+2}(z)-f_{k+1}(z)=0\\
    \label{eq:beta_line_psi_n}&C_\zeta = \{(x,0,\zeta)\mid x\in\mathbb{C}\}\text{ where }\zeta
    \text{ is a root of }f_{k+1}(z)-f_{k}(z)=0
\end{align}
For each of the components $C_z \subset X_{\psi_n}^{2} (S)$, the map  
$\overline{X}(N_n) \supset \overline{r}^{-1}q_2(C_z) \to \overline{X}_{\psi_n}(S)$ is a birational equivalence onto its image since the irreducible characters in \Cref{eq:alpha_line_psi_n,eq:beta_line_psi_n} arise from $\PSL(2,\mathbb{C})$--representations of $\pi_1(S)$ with trivial centraliser.

Points in $X_{\psi_n}(S)$ satisfy $2z=xy,$ and none of the roots $\xi$ or $\zeta$ equal zero. Hence the only point in $X_{\psi_n}(S)\cap X_{\psi_n}^{2} (S)$ is $(0,0,0),$ and the existence of this intersection point is due to the fact that we are not working with $\overline{X}(N_n),$ but with $\overline{X}_{\psi_n}(S).$ The two representations of $N_{n}$ predicted by \Cref{lem:K4_number_of_lifts} are given by letting
\begin{equation*}
    \overline{\rho}(a) = \pm\begin{pmatrix}
    0&1\\-1&0
    \end{pmatrix}
    \qquad\qquad
    \overline{\rho}(b) = \pm\begin{pmatrix}
    i&0\\0&-i
    \end{pmatrix}
\end{equation*}
and 
\[
T_1 =  \pm\begin{pmatrix}
    \frac{-i}{\sqrt{2}}&\frac{-1}{\sqrt{2}}\\ \frac{1}{\sqrt{2}}&\frac{i}{\sqrt{2}}
    \end{pmatrix}
      \qquad\qquad
      T_2  = \overline{\rho}(a) \; T_1 = 
      \pm\begin{pmatrix}
    \frac{1}{\sqrt{2}}&\frac{i}{\sqrt{2}}\\ \frac{i}{\sqrt{2}}&\frac{1}{\sqrt{2}}
    \end{pmatrix}
   \]
Then $\overline{\rho}$ extends to a representation into $\PSL(2,\mathbb{C})$ by letting either 
$\overline{\rho}=T_1$ or $\overline{\rho}=T_2,$ and the former lifts to $\SL(2,\mathbb{C})$ whilst the latter doesn't. Denote the characters in $\overline{X}(N_{n})$ of the $2$ extensions as $\overline{\chi_1}$ and $\overline{\chi_2}$. According to the discussion in \Cref{sec:basic_char}, $\overline{\chi_1}$ and $\overline{\chi_2}$ are on different topological components of $\overline{X}(N_{n}).$ Note that $\overline{\chi_1}$ is in $\overline{X_0}(N_n)$ while $\overline{\chi_2}$ is in $\overline{r}^{-1}q_2(L_3)$.

For the component $L_3\subset X_{\psi_n}^{2} (S)$, we follow the same parametrisation of $\overline{L_3}$ in \Cref{eq:param_of_PL3}. For any point $(0,0,(p+\frac{1}{p})^2,0)\in \overline{L_3}$, a direct computation shows that it extends to $2$ characters in $\overline{X}(N_n)$ where the quadruples in \Cref{eq:PSL_characters} are  $(2-p^{n+2}-\frac{1}{p^{n+2}},0,0,2-p^{n+4}-\frac{1}{p^{n+4}})$ and $(2+p^{n+2}+\frac{1}{p^{n+2}},0,0,2+p^{n+4}+\frac{1}{p^{n+4}})$. Thus $\overline{r}^{-1}q_2(L_3) \to \overline{L_3}$ is a $2$-fold branched cover with the only ramification point $\overline{\chi_2}$.


\subsection{The family $L_n$}
\label{sec:third_family}

Let $L_n$ be the once-punctured torus bundle with monodromy $\omega_n = A^2B^{n+2}A$. Note that $\tr (\omega_n) = -3n-4$ so $L_n$ is hyperbolic when $n$ is positive. When $n$ is odd, $b_1(\omega_n)=1$ and hence $\overline{X}^{irr}(L_n)\cong X_{\omega_n}(S)$ according to \Cref{cor:iso for b1=1}. We now show:

\begin{theorem}\label{thm: unbounded psl canonical curve}
Suppose $n=2k+1\geq 3$ is odd. The $\PSL(2,\mathbb{C})$--character variety $\overline{X}^{\text{irr}}(L_{n})$ consists of $k+2$ components. The canonical curve $\overline{X}_0(L_n)\subset \overline{X}^{\text{irr}}(L_{n})$ is birationally equivalent to a hyperelliptic curve of genus $k$ while all the other Zariski components are birational to the affine line. 
\end{theorem}

We remark that if $n=1$, then $\overline{X}^{\text{irr}}(L_{n})$ has $3$ components, each of genus zero.

The framing $\omega_n$ of $L_n$ is
\begin{equation*}
    {\omega_n} = \begin{cases}
    a\rightarrow a(a^{-2}b^{-1})^{n+2}\\
    b\rightarrow ba^3(a^{-2}b^{-1})^{n+2}
    \end{cases}
\end{equation*}

Then $X_{\omega_n}(S)$ is generated by
\begin{align*}
    \tr\rho(a) &= \tr \rho(a(a^{-2}b^{-1})^{n+2})\\
    \tr\rho(b) &= \tr \rho(ba^3(a^{-2}b^{-1})^{n+2})\\
    \tr\rho(ab) &= \tr \rho(a(a^{-2}b^{-1})^{n+2}ba^3(a^{-2}b^{-1})^{n+2})
\end{align*}
Using the trace identities in \Cref{sec:Preliminaries}, the second equation is $\tr\rho(b) = \tr \rho(a(a^{-2}b^{-1})^{n+1})$ and the third equation is given by $z = xy-zx + y$. Let $Q_n(x,y,z)\coloneqq \tr\rho(a(a^{-2}b^{-1})^{n})$. We have $Q_0=x,Q_1=z$ and $Q_{n} = (xz-y)Q_{n-1}-Q_{n-2}$. Note that the third equation is equivalent to $(y-z)(x+1)=0$ and hence naturally splits the variety $X_{\omega_n}(S)$ into the union of two algebraic sets as follows. 
\[X_{\omega_n}(S) = V(\langle Q_{n+2}-x,Q_{n+1}-y, x+1\rangle)\cup V(\langle Q_{n+2}-x,Q_{n+1}-y, y-z\rangle)\]
We denote the two algebraic sets in this union by $R_n$ and $W_n$ respectively.
\begin{lemma}
Each canonical component of $X^{irr}(L_n)$ is contained in the preimage of $W_n$.
\begin{proof}
We only need to show that each discrete and faithful character $\chi_{\rho_0}\in X^{irr}(L_n)$ is not in the preimage of $R_n$. Clearly $\rho_0$ is not binary dihedral, so we can write $\rho_0(a)$ and $\rho_0(b)$ as in \Cref{eq:standard_irrep}. If $x=-1$, a simple calculation shows that $\rho_0(a)$ is of order 3, which contradicts the faithfulness of $\rho_0$ since $a$ has infinite order in $\pi_1(M_{\omega_n}).$
\end{proof}
\end{lemma}
We will analyse $W_n$ and $R_n$ separately. Before doing so, we define a family of auxiliary polynomials $Q_n'\in\mathbb{C}[x,y,z]$, which follow the same recursive relation as $Q_n$ but with $Q_0' = x$ and $Q_1'=y$. When $n=2k+1$ is odd, a similar argument as in \Cref{lem: S-polynomials} shows that

\[\langle Q_{n+2}-Q_{0}',Q_{n+1}-Q_1'\rangle=\langle Q_{k+2}-Q_{k+1}',Q_{k+1}-Q_{k+2}'\rangle\]

In addition, we observe that 
\[Q_n = zf_n(xz-y)-xf_{n-1}(xz-y)\quad\text{ and }\quad Q_n' = yf_n(xz-y)-xf_{n-1}(xz-y)\]
\subsubsection{The canonical component}
Under the map $(x,y,z)\mapsto(x-1,y)$, $W_n$ is birationally equivalent with its image, denoted by $W_n'\subset\mathbb{C}^2$. As $Q_n$ and $Q_n'$  coincide when $y=z$, $W_n'$ is generated by the single element $q_{k+2}(x,y)-q_{k+1}(x,y)\in\mathbb{C}[x,y]$, where 
\[q_n(x,y) = yf_n(xy) - (x+1)f_{n-1}(xy)\] 
Let $\mathcal{W}_n$ be the variety in $\mathbb{C}^2$ generated by the polynomial 
\begin{align*}
    r_k(x,y) \coloneqq y^2(f_{k+2}(x)-f_{k+1}(x))-(x+y)(f_{k+1}(x)-f_{k}(x))
\end{align*}
There exists a birational isormorphism between $W_n'$ and $\mathcal{W}_n$ via 
\begin{align*}
&
\begin{aligned}
r_1:W_n'&\rightarrow \mathcal{W}_n\\
(x,y)&\rightarrow({xy},y)
\end{aligned}
&
\begin{aligned}
 r_2:\mathcal{W}_n&\rightarrow W_n'\\
     (x,y)&\rightarrow(\frac{x}{y},y)
\end{aligned}
\end{align*}

\begin{lemma}
$r_k(x,y)$ defined as above is an irreducible polynomial in $\mathbb{C}[x,y]$ for every positive odd integer $n=2k+1\geq 3$.
\begin{proof}
If $r_{k}(x,y)$ is reducible, then we can factorise it into $r_{k}(x,y) = s_k(x,y)t_k(x,y)$, where $t_n,s_n$ are non-constant polynomials in $\mathbb{C}[x,y]$. The Newton polygon of $r_k$ is shown in \Cref{fig: Newton polygon Ln}. The boundary vector sequence of $r_k(x,y)$ is
$$v(r_k) = \bigg( (1,0),...,(1,0),(0,1),(0,1),(-1,0),(-1,0),...,(-1,0),(0,-1),(1,-1)\bigg)$$ 
where the number or repetitions of of $(1,0)$ and $(-1,0)$ in the sequence are $k$ and $k+1$ respectively. Given $(0,1)$, $(0,1)$, $(0,-1)$ and $(1,-1)$ are the only vectors with nonzero horizontal component, there are two cases to split $v(r_k)$ into two disjoint nonempty subsequences $v(s_k)$ and $v(t_k)$, each of which sums to zero. \begin{enumerate}
    \item All the four vectors with nonzero horizontal component are in the same subsequence. Without loss of generality, let $s_k(x,y) = s_k(x)$ and $t_k(x,y) = y^2t^{(2)}_k(x)+yt^{(1)}_k(x)+t^{(0)}_k(x)$. 
    Then 
\begin{equation*}
\begin{cases}
    f_{k+2}(x) - f_{k+1}(x) = t^{(2)}_k(x)s_k(x)\\
    -(f_{k+1}(x) - f_{k}(x)) = t^{(1)}_k(x)s_k(x)\\
    -x(f_{k+1}(x) - f_{k}(x)) = t^{(0)}_k(x)s_k(x)\\
\end{cases}
\end{equation*}
This is impossible due to \Cref{lem: fibonacci poly}~(5).
\item Hence $(1,-1)$,$(0,1)$ belong to one sequence, and $(0,1)$,$(0,-1)$ belong to the other sequence. Let $r_k(x,y) = (a(x)y+b(x))(c(x)y+d(x))$ where $a,b,c,d$ are polynomials in $\mathbb{C}[x]$. Then
\begin{equation*}
    \begin{cases}
    f_{k+2}(x) - f_{k+1}(x) = a(x)c(x)\\
    -(f_{k+1}(x) - f_{k}(x)) = a(x)d(x)+b(x)c(x)\\
    -x(f_{k+1}(x) - f_{k}(x)) = b(x)d(x)
\end{cases}
\end{equation*}
Without loss of generality, take a root $\alpha$ of $f_{k+1}(x) - f_{k}(x)$ which is also a root of $b(x)$. 
The third equation and \Cref{lem: fibonacci poly}~(4) imply $d(\alpha) \neq 0$. 
The first equation and \Cref{lem: fibonacci poly}~(5) imply $a(\alpha)\neq 0$.
But this gives a contradiction to the second equation.
\end{enumerate}
Hence $r_k(x,y)$ is irreducible.
\end{proof}
\end{lemma}
\begin{remark}
When $n=1$, $W_1 = V(q_2-q_1) = V(y+1)\cup V(xy-x-1)$. Both components have genus $0$.
\end{remark}

\begin{figure}[h]
    \centering
    \begin{tikzpicture}
    \draw[very thin,color=gray] (-0.1,-0.1) grid (5.5,3);    
    \draw[->] (-0.2,0) -- (5.6,0) node[right] {$x$};
    \draw[->] (0,-0.2) -- (0,3.6) node[above] {$y$};

    \draw[color = blue](0,2)--(5,2)--(5,0)--(1,0)--(0,1)--(0,2);
    \node at (-0.5,2) {(0,0)};
    \node at (1,-0.5) {(1,0)};
    \node at (5,-0.5) {(k+1,0)};
    \node at (5,2.5) {(k+1,2)};
    \node at (-0.5,1) {(0,1)};
    \node at (0.3,0.3) {$\gamma_1$};
    \node at (3.5,-0.3) {$\gamma_2$};
    \node at (3.5,2.3) {$\gamma_4$};
    \node at (5.3,0.7) {$\gamma_3$};
    \node at (-0.3,1.5) {$\gamma_5$};
\end{tikzpicture}
    \caption{The Newton polygon of $r_k$}
    \label{fig: Newton polygon Ln}
\end{figure}
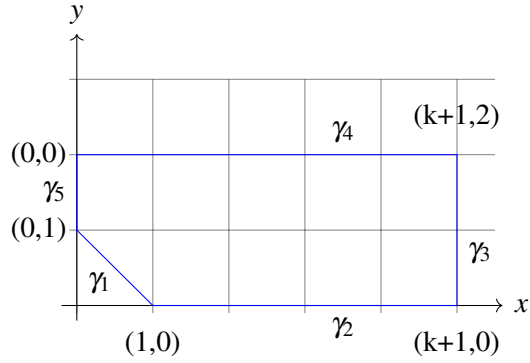

\begin{lemma}\label{lem: nondegeneracy Ln}
When $n=2k+1\geq 3$, $r_k$ is nondegenerate and there is no singular point in $V_n=V(r_k)$. Hence $V_k$ is a hyperelliptic curve of genus $k.$
\begin{proof}
The proof is straight forward using the same arguments as in the proof of \Cref{lem:singularity Un}.
\end{proof}
\end{lemma}

\subsubsection{Other components}
When $n=2k+1$, $R_n = V(Q_{k+2}-Q_{k+1}',Q_{k+1}-Q_{k+2}', x+1)$ is birationally equivalent with its image, denoted by $R_n'$, under the restriction $(x,y,z)\mapsto(y,z)$. The variety $R_n'$ is generated by $q_{k+2}'-q_{k+1}$ and $q_{k+2}-q_{k+1}'$, where 
\begin{align*}
    q_n(y,z) &= zf_n(-z-y)+f_{n-1}(-z-y)\\
    q_n'(y,z) &= yf_n(-z-y)+f_{n-1}(-z-y)
\end{align*}
By a direct computation using the definition of $f_n$, 
\[
    R_n = \{(-1,-1)\}\cup V(f_{k+2}(-y-z)+f_{k+1}(-y-z))
\]
The point $(-1,-1,-1)$ is already contained in $W_n$ while the second algebraic set gives us $k+1$ components 
\[\{(x,y,z)\mid y+z+\alpha=0,x+1=0\}\]
as $\alpha$ ranges over the $k+1$ distinct roots of $f_{k+2}(u) + f_{k+1}(u)=0$. All these components are birationally equivalent with affine lines.


\subsection{Experimental results}
\label{sec:computations}

As stated in \Cref{sec:prelim_bundles}, the choice of a monodromy as a positive word in $\A$ and $\B$ is not unique. Also, there  is no simple criterion to ensure that the absolute value of the trace is greater than two. We list in \Cref{tab:comp1} below the shortest words with this property. The computations were executed with \texttt{Singular}~\cite{Singular} using the following code template, and $\# X_\varphi(S)$ denotes the number of Zariski components of $X_\varphi(S).$ The map \texttt{m} below represents $\overline{\varphi}$ and is given as a composition of maps \texttt{alpha} and \texttt{betainv}.\\

\begin{lstlisting}
    ring r=0,(x,y,z),dp;
    map alpha=r,x,z,xz-y;
    map betainv=r,z,y,yz-x;
    map m = [composition of alpha and betainv];
    ideal I = m[1]-x,m[2]-y,m[3]-z;
    def S = absPrimdecGTZ(I);
    setring S;
    absolute_primes;
\end{lstlisting}

\begin{table}[h!]
\caption{Short words in $\A$ and $\B$ giving hyperbolic once-punctured torus bundles}
\begin{center}
\label{tab:comp1}
\begin{tabular}{| l | l | l | l | l |}
$\varphi_*$ & $\tr(\varphi_*)$ & $o(\varphi_2)$ & $\# X_\varphi(S)$ & genera \\
\hline
\hline
$A^2B^3$ & $-4$ 	& 2& 1	& 0	\\
\hline
$A^3B^2$ &$-4$ &	 2	& 1 & 0	\\
\hline
\hline
$AB^5$ & $-3$	& 3& 1	& 0	\\
\hline
$A^2B^4$ & $-6$	&1	& 2 & 0,0 	\\
\hline
$A^3B^3$ & $-7$	&3	& 3 & 0,0,0	\\
\hline
$A^4B^2$ & $-6$	& 1	& 2 & 0,0	\\
\hline
\end{tabular}
\end{center}
\end{table}

The representative
\[ \pm A^{a_1}B^{-b_1}A^{a_2}B^{-b_2} \cdots A^{a_n}B^{-b_n}\]
where $n>0$, the $a_i$ and $b_i$ are positive integers, and the sign equals the sign of the trace of $\varphi_*,$ allows us to build up a census of examples more efficiently. As explained in \Cref{sec:mutation}, if one is only interested in the topology of the fixed-point set, then it suffices to restrict to the case of positive trace. We summarise our computational results in \Cref{tab:comp2}.

\begin{table}[h!]
\caption{Short words in $A$ and $B^{-1}$ giving hyperbolic once-punctured torus bundles}
\begin{center}
\label{tab:comp2}
\begin{tabular}{| l | r | c | l | l |}
$\varphi_*$ & $\tr(\varphi_*)$ & $o(\varphi_2)$ & $\# X_\varphi(S)$ & genera \\
\hline
\hline
$AB^{-1}$ &3 & 3 &	 1	& 0	\\
\hline
\hline
$AB^{-2}$ &4&2 & 	1 	& 0	\\
\hline
$A^2B^{-1}$ &4& 2& 	 1	& 0	\\
\hline
\hline
$AB^{-3}$. & 5 &3 & 1		& 0	\\
\hline
$A^2B^{-2}$ & 6& 1& 2	 	& 0,0	\\
\hline
$A^3B^{-1}$& 5 &3 & 1	 	& 0	\\
\hline
$AB^{-1}AB^{-1}$& 7 &3  	& 3	& 0,0,0	\\
\hline
\hline
\hline
$AB^{-4}$ &6 &2 	& 2	& 0,0	\\
\hline
$A^2B^{-3}$ & 8& 2	& 1	& 1	\\
\hline
$A^3B^{-2}$ &8 &2 	& 1	& 1	\\
\hline
$A^4B^{-1}$ & 6& 2	& 2	& 0,0	\\
\hline
$AB^{-1}AB^{-2}$ & 10& 2	& 3	& 0,0,0	\\
\hline
$AB^{-1}A^2B^{-1}$ &10 & 2	& 3	& 0,0,0	\\
\hline
\hline
\hline
$AB^{-5}$ & 7&3 	& 1	& 0	\\
\hline
$A^2B^{-4}$ &10&1 & 	2 	& 1,0	\\
\hline
$A^3B^{-3}$ &11 &3 	& 2	& 1,0	\\
\hline
$A^4B^{-2}$ &10 &1 	& 2	& 1,0	\\
\hline
$A^5B^{-1}$ &7 & 3	& 1	& 0	\\
\hline
$AB^{-1}AB^{-3}$ &13 &3 	& 3	& 1,0,0	\\
\hline
$AB^{-1}A^2B^{-2}$ &15 &3 	& 1	& 2	\\
\hline
$AB^{-1}A^3B^{-1}$ &13 &3 	& 3	& 1,0,0	\\
\hline
$AB^{-2}AB^{-2}$ & 14& 1	& 5	& 0,0,0,0,0	\\
\hline
$AB^{-2}A^2B^{-1}$ &15 &3 	& 1	& 2	\\
\hline
$AB^{-1}AB^{-1}AB^{-1}$ &18 &1 	& 7	& 0,0,0,0,0,0,0	\\
\hline
\hline

\end{tabular}
\end{center}
\end{table}



\bibliographystyle{plain}
\bibliography{references}

\address{Stephan Tillmann\\School of Mathematics and Statistics F07, The University of Sydney, NSW 2006 Australia\\{stephan.tillmann@sydney.edu.au\\-----}}

\address{Youheng Yao\\School of Mathematics and Statistics F07, The University of Sydney, NSW 2006 Australia\\{yyao3610@uni.sydney.edu.au}}

\Addresses                                       
\end{document}